\documentclass[11pt,reqno]{amsart}
\usepackage[utf8]{inputenc}
\usepackage{amssymb}
\usepackage[all]{xy}
\usepackage{xcolor}
\usepackage{hyperref}
\allowdisplaybreaks

\newtheorem{thm}{Theorem}[section]
\newtheorem{lem}[thm]{Lemma}
\newtheorem{prop}[thm]{Proposition}
\newtheorem{cor}[thm]{Corollary}

\newtheorem{lem-def}[thm]{Lemma-Definition}

\theoremstyle{definition}
\newtheorem{dfn}[thm]{Definition}

\theoremstyle{remark}
\newtheorem{remark}[thm]{Remark}

\newcommand{\CA}{{\mathcal{A}}}

\newcommand{\CF}{{\mathcal{F}}}
\newcommand{\CG}{{\mathcal{G}}}

\newcommand{\CI}{{\mathcal{I}}}

\newcommand{\CL}{{\mathcal{L}}}

\newcommand{\CZ}{{\mathcal{Z}}}
\newcommand{\CB}{{\mathcal{B}}}

\newcommand{\af}{\alpha}
\newcommand{\bt}{\beta}
\newcommand{\gm}{\gamma}
\newcommand{\dt}{\delta}

\newcommand{\sm}{\sigma}

\newcommand{\Z}{{\mathbb{Z}}}

\newcommand{\C}{{\mathbb{C}}}
\newcommand{\N}{{\mathbb{N}}}
\newcommand{\T}{{\mathbb{T}}}

\newcommand{\PDS}{\text{\textsc{PDynSys}$_\mathrm{Orb}$}}
\newcommand{\TG}{\text{\textsc{TopGrpd}}}
\newcommand{\Hom}{\operatorname{Hom}}
\newcommand{\scj}{\subseteq}

\makeatletter
\@namedef{subjclassname@2020}{%
  \textup{2020} Mathematics Subject Classification}
\makeatother

\begin{document}


\title[ Continuous Orbit Equivalence for Partial Dynamical Systems ]
{A Categorical Interpretation of Continuous Orbit Equivalence for Partial Dynamical Systems}

\author[Gilles G. de Castro]{Gilles G. de Castro}
\address{Departamento de Matem\'atica, Universidade Federal de Santa Catarina, 88040-970 Florian\'opolis SC, Brazil.} \email{gilles.castro@ufsc.br}

\author[E. J. Kang]{Eun Ji Kang}
\address{Research Institute of Mathematics, Seoul National University, Seoul 08826, 
Korea} \email{kkang3333\-@\-gmail.\-com}

\thanks{This research was supported by Basic Science Research Program through the 
National Research Foundation of Korea(NRF) funded by the Ministry of Education(RS-2023-00238961) and the Ministry of Science and ICT(NRF-2022M3H3A1098237).}



\begin{abstract}

We define the orbit morphism of partial dynamical systems and prove that an orbit morphism being an isomorphism in the category of partial dynamical systems and orbit morphisms is equivalent to the existence of a continuous orbit equivalence between the given partial dynamical systems that preserves the essential stabilisers. We show that this is equivalent to the existence of a diagonal-preserving isomorphism between the corresponding crossed products when the essential stabilisers of partial actions are torsion-free and abelian. We also characterize when an étale groupoid is isomorphic to the transformation groupoid of some partial action. Additionally, we explore the implications in the context of semi-saturated orthogonal partial dynamical systems over free groups, establishing connections with Deaconu-Renault systems and the concept of eventual conjugacy. Finally, we apply our results to C*-algebras associated with generalized Boolean dynamical systems.
 \end{abstract}

\maketitle

\section{Introduction}

The interplay between topological dynamical systems and $C^*$-algebras plays a crucial role in understanding the underlying dynamical and algebraic structures, providing key insights into both fields of study.
A notable result demonstrating this connection is the equivalence between orbit equivalence of minimal homeomorphisms on  Cantor sets and diagonal-preserving isomorphisms of their corresponding crossed products (\cite{GPS}). 
This result has been extended to topologically free homeomorphisms on compact Hausdorff spaces (\cite{T1996, BT1998}), among many others, and further generalized to dynamical systems arising from topologically free group actions on locally compact Hausdorff spaces (\cite[Theorem 1.2]{Li2018}).
In recent years, building on these developments, the concept of stabiliser-preserving continuous orbit equivalence was introduced for group actions on such spaces, and it was subsequently shown that if the essential stabilisers of group actions are torsion-free and abelian, the actions are classified, up to stabiliser-preserving continuous orbit equivalence, by diagonal-preserving isomorphisms of their associated crossed products (\cite[Corollary 7.5]{CRST}).
 
The notion of a partial dynamical system has emerged as an important framework in the study of dynamical systems, particularly because many $C^*$-algebras can be naturally constructed as crossed products associated with partial dynamical systems, while the framework of ordinary dynamical systems tends to be more limited (see, for instance, \cite{Dok1,ExelBook}). For partial dynamical systems, the notion of continuous orbit equivalence was introduced in \cite{Li2017}, where the relationship with diagonal-preserving isomorphisms is established for topological free actions \cite[Theorem 2.7]{Li2017}, that is to say, that essential stabilisers groups are all trivial. One of the main goals of this paper is to extend both \cite[Corollary 7.5]{CRST} and \cite[Theorem 2.7]{Li2017}, which is achieved in Theorem \ref{isom-equivalences} and Proposition \ref{prop:isom-equivalences}. A crucial step in both results is to understand when the transformation groupoids of two actions (or partial actions) are isomorphic in terms of continuous orbit equivalence.

To tackle this problem, we first introduce the notion an orbit morphism between partial dynamical systems to consider the category where the objects are partial dynamical systems and the morphisms are orbit morphisms. We then construct a fully faithful functor from this category  to the category of topological groupoids and continuous homomorphisms (Theorem \ref{thm: functor}). As a result, we obtain that two partial dynamical systems are isomorphic in the category of partial dynamical systems and orbit morphisms if and only if the corresponding transformation groupoids are isomorphic as groupoids. We also characterize when an étale groupoid is isomorphic to some transformation groupoid of a partial dynamical system (Theorem~\ref{thm:transf.groupoid}\footnote{The first author would like to thank Alcides Buss and Ruy Exel for the discussion that led to this result.}). We then prove that the transformation groupoids being isomorphic is equivalent to the existence of a continuous orbit equivalence between the given partial dynamical systems that preserves the essential stabilisers  (Theorem \ref{isom-equivalences}). This approach not only broadens the applicability of the results but also highlights the strength of the underlying structures involved in the analysis of $C^*$-algebras associated with dynamical systems. Additionally, we utilize established results (\cite[Theorem 6.2]{CRST}) from groupoid theory to show that if the essential stabilisers of partial actions are torsion-free and abelian, then  these are equivalent to the existence of a diagonal-preserving isomorphism between the corresponding crossed products (Proposition \ref{prop:isom-equivalences}).

Many classes of C*-algebras can be described using a partial action of the free group, such as Exel-Laca algebras \cite{ExelLaca1999}, graph C*-algebras \cite{CarlsenLarsen2016}, labelled graph C*-algebras \cite{CastroWyk}, C*-algebras of generalized Boolean dynamical systems \cite{CasK2} and subshift C*-algebras \cite{BCGW2023}. We study, in general, semi-saturated orthogonal partial dynamical systems over a free group and define a Deaconu-Renault system based on them. We demonstrate that our results interpreted in this context are equivalent to the essential-stabiliser-preserving continuous orbit equivalence of the Deaconu-Renault system (Corollary \ref{isom-equivalences:oss}).
We also define the concept of eventual conjugacy between semi-saturated orthogonal partial dynamical systems over free groups and prove that this is equivalent to the corresponding Deaconu-Renault systems being eventually conjugate. Regarding crossed products corresponding to semi-saturated orthogonal partial dynamical systems over free groups, there is a natural action of the circle on the crossed product. We furthermore prove that the eventual conjugacy of these partial actions is equivalent to the existence of a diagonal-preserving isomorphism that commutes with actions between the crossed products (Corollary \ref{equivalences ec}).

We structure this paper as follows. In Section \ref{preliminary}, we present the necessary background. In Section \ref{Sec 3}, we define the notion of continuous orbit equivalence in the context of partial dynamical systems from a categorical perspective and derive key results related to it. In Section \ref{Sec 4}, we focus our attention on semi-saturated orthogonal partial dynamical systems over free groups. Finally, in Section \ref{Sec 5}, we provide an example to illustrate our findings.

\section{Preliminaries}\label{preliminary}

\subsection{Groupoids} 
Let $\CG$ be a locally compact Hausdorff groupoid with unit space $\CG^{(0)}$. We say that $\CG$ is {\it \'etale} if its range and source maps $r,s: \CG \to \CG^{(0)}$ are local homeomorphism. 
 
For a unit  $u \in \CG^{(0)}$, we define the following sets: 
$$\CG^u:=r^{-1}(u), ~\CG_u:=s^{-1}(u) ~\text{and}~\CG^u_u:=r^{-1}(u) \cap s^{-1}(u).$$
The {\it isotropy subgroupoid} of $\CG$ is $$\operatorname{Iso}(\CG):=\bigcup_{u \in \CG^{(0)}}\CG_u^u=\{\gm \in \CG: s(\gm)=r(\gm)\}.$$
For each $u \in \CG^{(0)}$, the set $\CG^u_u$ is called the {\it isotropy group} at $u$. 
  Let $\operatorname{Iso}(\CG)^{\circ}$ denote the interior of $\operatorname{Iso}(\CG)$ within $\CG$. If $\CG$ is \'etale, then we have $\CG^{(0)} \subseteq \operatorname{Iso}(\CG)^{\circ}$ and $\operatorname{Iso}(\CG)^{\circ}$ forms an  \'etale subgroupoid of $\CG$.

A groupoid is called {\it torsion-free} if for every $u \in \CG^{(0)}$, the isotropy group $\CG_u^u$ is torsion-free.

\subsection{Partial dynamical systems} 
 In what follows, all groups are assumed to be discrete and countable, and topological spaces are locally compact, Hausdorff, and second-countable.
\begin{dfn} 
A {\it partial action} of a group $G$ with identity $e_G$ on a topological space $X$ consists of a  collection $\{U_g\}_{g \in G}$ of open subsets of $X$  and a collection $\{\varphi_g\}_{g \in G}$ of homeomorphisms $$\varphi_g: U_{g^{-1}} \to U_g, \ x \mapsto g.x$$
such that 
\begin{enumerate}
\item $U_{e_G}=U_{e_{G}^{-1}}=X$ and $\varphi_{e_G}=id_X$,
\item $g_2.(U_{(g_1g_2)^{-1}} \cap U_{g_2^{-1}})=U_{g_2} \cap U_{g_1^{-1}}$ for all $g_1, g_2 \in G$, and
$(g_1g_2). x=g_1. (g_2.x)$ for all $x \in U_{g_2^{-1}} \cap U_{(g_1g_2)^{-1}}$.
\end{enumerate}
The triple $(X,G,\varphi)$ is referred to as a \textit{partial dynamical system}, and we write this  $\varphi: G \curvearrowright X$ or simply $G \curvearrowright X$. Given a partial dynamical system $G \curvearrowright X$, for each $x\in X$, we define the set $G_x:=\{g\in G:x\in U_{g^{-1}}\}.$
\end{dfn}

\begin{remark} Let $\{U_g\}_{g \in G}$ be a collection of open subsets of $X$, and let 
$\{\varphi_g\}_{g \in G}$  be a collection of 
homeomorphisms $\varphi_g: U_{g^{-1}} \to U_g, \ x \mapsto g.x$.
Then, the equality  $$g_2.(U_{(g_1g_2)^{-1}} \cap U_{g_2^{-1}})=U_{g_2} \cap U_{g_1^{-1}}$$ holds  for all $g_1, g_2 \in G$ if and only if  the equality
$$g_1.(U_{g_1^{-1}} \cap U_{g_2})=U_{g_1} \cap U_{g_1g_2}$$
holds 
 for all $g_1,g_2 \in G$.
\end{remark}

\begin{proof} ($\Rightarrow$) With the change of variables $g_1:=(g_1g_2)^{-1}$ and $g_2:=g_1$, we have 
$$g_1.( U_{g_2} \cap U_{g_1^{-1}} )  =g_1.( U_{((g_1g_2)^{-1}g_1)^{-1}} \cap U_{g_1^{-1}} )=U_{g_1} \cap U_{g_1g_2}.$$

($\Leftarrow$) With the change of variables $g_1:=g_2$ and $g_2:=(g_1g_2)^{-1}$, we have 
$$g_2.(U_{g_2^{-1}} \cap U_{(g_1g_2)^{-1}})=U_{g_2} \cap U_{g_2(g_1g_2)^{-1}}=U_{g_2} \cap U_{g_1^{-1}}.$$
\end{proof}
The transformation groupoid associated with the partial dynamical system $\varphi: G \curvearrowright X$
 is defined by
$$G \ltimes_\varphi X:=\{(g,x) \in G \times X: g\in G ~\text{and}~x \in U_{g^{-1}}\},$$
where the source map is given by $s(g,x)=x$, the range map is given by $r(g,x)=g.x$, 
the composition of elements is defined by $(g_1, g_2.x)(g_2,x)=(g_1g_2, x)$, and the inverse of an element $(g,x)$  is $(g^{-1}, g.x)$. 
We usually write $G \ltimes X$ for $G \ltimes_\varphi X$ when the action $\varphi$ is clear from the context.

\begin{lem} Let $ G \curvearrowright X$ be a partial dynamical system and let $x \in X$. Then, 
\begin{enumerate}
\item The set $\operatorname{Stab}(x):=\{ g \in G: x \in U_{g^{-1}} ~\text{and}~ g.x=x\}$ is a subgroup of $G$.
\item The set $\operatorname{Stab}^{ess}(x):=\{g \in  G: g \in \operatorname{Stab}(y) ~\text{for all}~ y ~\text{in some neighborhood}~ \allowbreak V \text{of}~x\}$ is a subgroup of $G$.
\end{enumerate}
\end{lem}

\begin{proof}(1): Let $g_1, g_2 \in \operatorname{Stab}(x)$. Then, $x \in U_{g_1^{-1}} \cap U_{g_2^{-1}}$ and $g_1.x=x=g_2.x$. Then, since $x \in U_{g_1}\cap U_{g_2}$ and  $g_2^{-1}.x=g_2^{-1}. (g_2.x)=(g_2^{-1} g_2).x=e_G.x=x$, we have $x=g_2^{-1}.x \in g_2^{-1}.(U_{g_2} \cap U_{g_1^{-1}})=U_{g_2^{-1}} \cap U_{g_2^{-1}g_1^{-1}}$. So, $x \in U_{(g_1g_2)^{-1}}$ and $(g_1g_2).x=g_1.(g_2.x)=g_1.x=x$. Thus, $g_1g_2 \in \operatorname{Stab}(x)$.
 Also, clearly, $x \in U_{g_1}$, and we have $g_1^{-1}.x=g_1^{-1}.(g_1.x)=(g_1^{-1}g_1).x=e_G.x=x $. Thus, $g_1^{-1} \in \operatorname{Stab}(x)$. Hence, $\operatorname{Stab}(x)$ is a subgroup of $G$.
 
 (2): Let $g_1, g_2 \in \operatorname{Stab}^{ess}(x)$. Then, there are open neighborhood $V_1$, $V_2$ such that $g_1 \in \operatorname{Stab}(y)$ for all $y \in V_1$ and $g_2 \in \operatorname{Stab}(y)$ for all $y \in V_2$. Then, by (1), we conclude that $g_1g_2^{-1} \in \operatorname{Stab}(y)$ for all $y \in V_1 \cap V_2$. Thus, $g_1g_2^{-1} \in \operatorname{Stab}^{ess}(x)$. Hence, $\operatorname{Stab}^{ess}(x)$ is a subgroup of $G$.
\end{proof}

We call $\operatorname{Stab}(x)$ the {\it stabiliser subgroup} at $x$ in $G$ and $\operatorname{Stab}^{ess}(x)$ the {\it essential stabiliser subgroup} of $x$ in $G$. It is easy to see that $\operatorname{Iso}(G \ltimes X)=\bigcup_{x \in X}  \operatorname{Stab}(x) \times \{x\}$ and that $\operatorname{Iso}(G \ltimes X)^{\circ}=\bigcup_{x \in X}  \operatorname{Stab}^{ess}(x) \times \{x\}$.

\begin{lem}\label{partial system:torsion free and abelian} Let   $G \curvearrowright X$ be a partial dynamical system. Then,  $\operatorname{Stab}^{ess}(x)$ is torsion-free and abelian  for all $x \in X$ if and only if $\operatorname{Iso}(G \ltimes X)^{\circ}$ is  torsion-free  and abelian.
\end{lem}

\begin{proof} ($\Rightarrow$): Let $(g,x) \in \operatorname{Iso}(G \ltimes X)^{\circ}$ with $g \neq e_G$. Then, $g \in \operatorname{Stab}^{ess}(x)$. Since $\operatorname{Stab}^{ess}(x)$ is torsion-free, $g^m \neq e_G$ for all $m \geq 1$. So, $(g,x)^m=(g^m,x) \neq (e_G, x)$ for all $m \geq 1$. Thus, $\operatorname{Iso}(G \ltimes X)^{\circ}$ is  torsion-free.

Let $(g_1,x), (g_2,x) \in \operatorname{Iso}(G \ltimes X)^{\circ}$. Since  
$\operatorname{Stab}^{ess}(x)$ is  abelian, we have
$$(g_1,x)(g_2,x)=(g_1g_2,x)=(g_2g_1,x)=(g_2,x) (g_1,x).$$
Thus, $\operatorname{Iso}(G \ltimes X)^{\circ}$ is abelian.

 ($\Leftarrow$): Fix $x \in X$. First, choose $g (\neq e_G) \in \operatorname{Stab}^{ess}(x)$. Then, since $(g,x) \in \operatorname{Iso}(G \ltimes X)^{\circ}$ and $\operatorname{Iso}(G \ltimes X)^{\circ}$ is  torsion-free, $(g,x)^m=(g^m,x) \neq (e_G,x)$ for all $m \geq 1$. Thus, $g^m \neq e_G$ for all $m \geq 1$. Hence, $\operatorname{Stab}^{ess}(x)$ is torsion-free  for all $x \in X$.
 
 Secondly, let $g_1, g_2 \in  \operatorname{Stab}^{ess}(x)$. Then, since $(g_1,x), (g_2,x) \in \operatorname{Iso}(G \ltimes X)^{\circ}$ and $\operatorname{Iso}(G \ltimes X)^{\circ}$ is  abelian, we have
 $$(g_1g_2,x)=(g_1,x)(g_2,x)=(g_2,x) (g_1,x)=(g_2g_1,x).$$
So, $g_1g_2=g_2g_1$. Thus, $\operatorname{Stab}^{ess}(x)$ is abelian. 
\end{proof}

\subsection{Deaconu-Renault systems} A {\it Deaconu-Renault system} is defined as a pair $(X, \sigma)$, where $X$ is a locally compact Hausdorff space and $\sigma$ is a local homeomorphism that maps an open set $\operatorname{dom}(\sm) \subseteq X$ to another open set $\operatorname{ran}(\sm) \subseteq X$. We inductively define $\operatorname{dom}(\sm^n):=\sm^{-1}(\operatorname{dom}(\sm^{n-1}))$. 
So, each $\sm^n:\operatorname{dom}(\sm^n) \to \operatorname{ran}(\sm^n)$ is a local homeomorphism and $\sm^m \circ \sm^n=\sm^{m+n}$ on $\operatorname{dom}(\sm^{m+n})$.

The Deaconu-Renault groupoid of $(X, \sigma)$ is 
$$\Gamma(X,\sigma)=\bigcup_{n,m \in \N}\{(x,n-m,y) \in \operatorname{dom}(\sm^{n}) \times \{n-m\}\times \operatorname{dom}(\sm^{m}): \sigma^n(x)=\sigma^m(x)\}$$
equipped with the topology generated by the basic open sets $$Z(U,n,m,V)=\{(x,n,m,y): x \in U, y \in V ~\text{and}~ \sm^n(x)=\sm^m(y)\},
$$  where $n,m \in \N$, $U \subseteq \operatorname{dom}(\sm^{n})$ and $V \subseteq \operatorname{dom}(\sm^{m})$ are open, and $\sm^n|_U$ and $\sm^m|_V$ are homeomorphisms. 
The $\Gamma(X,\sigma)$ is a locally compact Hausdorff \'etale amenable groupoid with unit space $\{(x,0,x): x \in X \}$ identified with $X$.

Let $(X, \sigma)$ be a Deaconu-Renault system and $x \in X$. The {\it stabiliser group} at $x$ is defined by 
$$\operatorname{Stab}(x)=\{m-n: m,n \in \mathbb{N}, \ x \in \operatorname{dom}(\sm^{m}) \cap \operatorname{dom}(\sm^n), ~\text{and}~ \sm^m(x)=\sm^n(x)\} \subseteq \mathbb{Z},$$
and the {\it essential stabiliser group} at $x$ is defined by
\begin{align*} \operatorname{Stab}^{\operatorname{ess}}(x)=\{m-n: \ m,&\ n \in \mathbb{N}  ~\text{and there is an open neighborhood} \\
& ~U \subseteq \operatorname{dom}(\sm^{m}) \cap \operatorname{dom}(\sm^{n}) ~\text{of}~x ~\text{such that}~  \sm^m|_U=\sm^n|_U\}.
\end{align*}
Note that $\operatorname{Stab}^{\operatorname{ess}}(x) \subseteq \operatorname{Stab}(x)$ for each $x \in X$. 
Also, the {\it minimal stabiliser} of $x$ is defined to be 
$$\operatorname{Stab}_{\min}(x)=\min\{n \in\operatorname{Stab}(x) : n \geq 1\},$$
and the {\it minimal essential stabiliser} at $x$ is defined to be 
$$\operatorname{Stab}_{\min}^{\operatorname{ess}}(x)=\min\{n \in\operatorname{Stab}^{\operatorname{ess}}(x) : n \geq 1\},$$
where we define 
 $\min(\emptyset)=\infty$.

The following definition will be used in Corollary \ref{isom-equivalences:oss} and Corollary \ref{equivalences ec}.
\begin{dfn} Let $(X, \sigma)$ and $(Y, \tau)$ be Deaconu-Renault systems.
We say that $(X, \sigma)$ and $(Y, \tau)$ are {\it continuous orbit equivalent} (\cite[Definition 8.1]{CRST}) if there exist a homeomorphism $\phi: X \to Y$ and continuous functions $k, l: \operatorname{dom}(\sigma) \to \mathbb{N}$ and $k',l':\operatorname{dom}(\tau) \to \mathbb{N}$ such that 
$$\tau^{l(x)}(\phi(x))=\tau^{k(x)}(\phi(\sigma(x))) ~\text{and}~ \sigma^{l'(y)}(\phi^{-1}(y))=\sigma^{k'(y)}(\phi^{-1}(\tau(y)))$$
for all $x \in \operatorname{dom}(\sigma)$ and $y \in \operatorname{dom}(\tau) $.  We call $(\phi, k,l,k',l')$ {\it a continuous orbit equivalence} and call $\phi$ the {\it underlying homeomorphism}. 
 We say that $(\phi, k,l,k',l')$ {\it preserve stabilisers} if $\operatorname{Stab}_{\min}(x) < \infty \iff \operatorname{Stab}_{\min}(\phi(x)) < \infty $, and
\begin{align*} \left|\sum_{n=0}^{\operatorname{Stab}_{\min}(\mu)-1} l(\sm^n(x))-k(\sm^n(x))\right| &= \operatorname{Stab}_{\min}(\phi(x)) ~\text{and}~ \\
\left|\sum_{n=0}^{\operatorname{Stab}_{\min}(y)-1} l'(\tau^n(y))-k'(\tau^n(\nu))\right| &= \operatorname{Stab}_{\min}(\phi^{-1}(y))
\end{align*}
whenever $\operatorname{Stab}(x) \neq \{0\}$,  $\operatorname{Stab}(y) \neq \{0\}$, $\sm^{\operatorname{Stab}_{\min}(x)}(x)=x$, and $\tau^{\operatorname{Stab}_{\min}(y)}(y)=y$. 

Likewise, we say that  $(\phi, k,l,k',l')$ {\it preserve essential stabilisers} if 
 $\operatorname{Stab}_{\min}^{\operatorname{ess}}(x) < \infty \iff \operatorname{Stab}_{\min}^{\operatorname{ess}}(\phi(x)) < \infty $, and 
\begin{align*} \left|\sum_{n=0}^{\operatorname{Stab}_{\min}^{\operatorname{ess}}(x)-1} l(\sm^n(x))-k(\sm^n(x))\right| &= \operatorname{Stab}_{\min}^{\operatorname{ess}}(\phi(x)) ~\text{and}~ \\
\left|\sum_{n=0}^{\operatorname{Stab}_{\min}^{\operatorname{ess}}(y)-1} l'(\tau^n(y))-k'(\tau^n(y))\right| &= \operatorname{Stab}_{\min}^{\operatorname{ess}}(\phi^{-1}(y))
\end{align*}
whenever $\operatorname{Stab}_{\min}^{\operatorname{ess}}(x) \neq \{0\}$, $\operatorname{Stab}_{\min}^{\operatorname{ess}}(y)\neq \{0\}$, $\sm^{\operatorname{Stab}_{\min}^{\operatorname{ess}}(x)}(x)=x$, and $\tau^{\operatorname{Stab}_{\min}^{\operatorname{ess}}(y)}(y)=y$. 
\end{dfn}

\begin{dfn}
Let $(X, \sigma)$ and $(Y, \tau)$ be Deaconu-Renault systems. We say that $(X, \sigma)$ and $(Y, \tau)$ are {\it eventually conjugate} (\cite[Definition 8.9]{CRST}) if there is a stabiliser-preserving continuous orbit equivalence $(\phi, k,l,k',l')$ such that $l(x)=k(x)+1$ for all $x \in X$.
\end{dfn}

\section{Orbit morphisms between partial dynamical systems}\label{Sec 3}

In this section, we define a notion of orbit morphism between partial dynamical systems as well as a composition between theses morphism in order to obtain a category, which we denote by $\PDS$. We then build a fully faithful functor $F:\PDS\to\TG$, where $\TG$ is the category of topological groupoids with continuous homomorphisms (ie, continous functors) as morphisms. At the level of objects, the functor sends a partial dynamical system to its transformation groupoid.

\begin{dfn}\label{def orbit morphism}
    Let $G \curvearrowright X$ and $H \curvearrowright Y$ be partial dynamical systems. An {\em orbit morphism} between $G \curvearrowright X$ and $H \curvearrowright Y$ is a pair $(\phi,a)$, where $a: \bigcup_{g \in G} \{g\}\times U_{g^{-1}} \to H$ and $\phi:X\to Y$ are continuous functions such that:
    \begin{enumerate}
        \item\label{om preserves action} $\phi(x)\in V_{a(g,x)^{-1}}$ (where $V_{h^{-1}}$ is the domain of the partial homeomorphism attached to $h \in H$) and $\phi(g.x)=a(g,x).\phi(x)$, for all $g\in G$ and $x\in U_{g^{-1}}$,
        \item\label{om is a cocyle} $a(g_1g_2, x)=a(g_1,g_2.x)a(g_2,x)$ for all $g_1,g_2 \in G$ and $x \in U_{g_2^{-1}}\cap U_{(g_1g_2)^{-1}}$.
    \end{enumerate}
  We call the map $a$  a {\it cocycle} and write $(\phi,a): G \curvearrowright X \to H \curvearrowright Y$ to denote the orbit morphism $(\phi,a)$ between $G \curvearrowright X$ and $H \curvearrowright Y$.
\end{dfn}

\begin{remark} Let  $a: \bigcup_{g \in G} \{g\}\times U_{g^{-1}} \to H$ be a map. If $a$ is a cocycle, then we have $a(e_G,x)=a(e_Ge_G,x)=a(e_G, e_G.x)a(e_G,x)=a(e_G,x)^2$ for all $x \in X$, from which it follows that $a(e_G,x)=e_H$ for all $x \in X$. 
\end{remark}

\begin{remark}
Condition \eqref{om preserves action} of Definition~\ref{def orbit morphism} implies that, given an orbit morphism $(\phi,a)$ from $G\curvearrowright X$ to $H\curvearrowright Y$, for each $x\in X$ the map $a(\cdot,x):G_x\to H_{\phi(x)}$ is well-defined.
\end{remark}

The following lemma shows that the natural candidate for a composition of orbit morphisms is indeed an orbit morphism.

\begin{lem-def}
    Given orbit morphism $(\phi,a)$ between $G \curvearrowright X$ and $H \curvearrowright Y$, and $(\psi,b)$ between $H \curvearrowright Y$ and $K \curvearrowright Z$, consider the pair $(\psi\circ\phi,c)$, where $c:\bigcup_{g \in G} \{g\}\times U_{g^{-1}} \to K$ is given by $c(g,x)=b(a(g,x),\phi(x))$. Then $(\psi\circ\phi,c)$ is an orbit morphism between $G \curvearrowright X$ and $H \curvearrowright Y$. We then define the composition $(\psi,b)\circ (\phi,a)$ as $(\psi\circ\phi,c)$.
\end{lem-def}

\begin{proof}
    Note that $c$ is well-defined by Condition~\eqref{om preserves action} of Definition~\ref{def orbit morphism}.
    
    (1) Let $g\in G$ and $x\in U_{g^{-1}}$. We have that $\psi(\phi(x))\in W_{c(g,x)^{-1}}$ because $\phi(x)\in V_{a(g,x)^{-1}}$ and $(\psi,b)$ is an orbit morphism. Moreover,
    \begin{align*} \psi \circ \phi(g.x) &=\psi(\phi(g.x)) \\
     &=\psi(a(g,x).\phi(x)) \\
     &=b(a(g,x), \phi(x)).\psi(\phi(x)) \\
     &=c(g,x).\psi(\phi(x)) \\
     &=c(g,x).(\psi\circ\phi)(x).
     \end{align*}     

    (2) Let $g_1,g_2\in G$ and $x\in U_{g_2^{-1}}\cap U_{(g_1g_2)^{-1}}$. Then $\phi(x)\in V_{a(g_2,x)^{-1}}\cap V_{a(g_1g_2,x)^{-1}}$ and     
     \begin{align*} c(g_1g_2,x) &=b(a(g_1g_2,x), \phi(x)) \\
     &=b(a(g_1,g_2.x)a(g_2,x), \phi(x))\\
     &=b(a(g_1, g_2.x),a(g_2,x).\phi(x) )b(a(g_2,x), \phi(x)) \\
     &=b(a(g_1, g_2.x),\phi(g_2.x))c(g_2,x)\\
     &=c(g_1, g_2.x)c(g_2,x).
     \end{align*}
     
     It follows that $(\psi\circ\phi,c)$ is an orbit morphism.
\end{proof}

\begin{lem}
    The composition between orbit morphisms is associative and has identities. More precisely, $id_{G \curvearrowright X}=(id_X,i)$, where $i(g,x)=g$ for all $g\in G$ and $x\in U_{g^{-1}}$.
\end{lem}

\begin{proof}
    (Associativity): For $(\phi,a): G \curvearrowright X \to H \curvearrowright Y$, $(\psi,b): H \curvearrowright Y \to K \curvearrowright Z$ and $(\pi,d):  K \curvearrowright Z \to P \curvearrowright W$,  
     $$(\pi,d) \circ ((\psi,b)\circ(\phi,a))=(\pi,d) \circ (\psi\circ\phi,c)=(\pi \circ \psi \circ \phi, d'),$$
     where $d'(g,x)=d(c(g,x), \psi \circ \phi(x))=d(b(a(g,x), \phi(x)) , \psi \circ \phi(x))$.
     Also, 
     $$((\pi,d) \circ (\psi,b))\circ(\phi,a)=(\pi \circ \psi,e) \circ (\phi,c)=(\pi \circ \psi \circ \phi, e'),$$      where $e'(g,x)=e(c(g,x), \phi(x))=d(b(a(g,x), \phi(x)), \psi \circ \phi(x)))$.
     
     Thus, $(\pi,d) \circ ((\psi,b)\circ(\phi,a))=((\pi,d) \circ (\psi,b))\circ(\phi,a)$.
 
     (Identities): For  given an orbit morphism $(\phi,a)$ between $G \curvearrowright X$ and $H \curvearrowright Y$, 
     $$(\phi,a)\circ (id_X, i) =(\phi \circ id_X, c),$$
     where $c(g,x)=a(i(g,x), id_X(x))=a(g, x)$. Thus,$ (\phi,a)\circ (id_X, i) =(\phi,a)$. Also, for  given an orbit morphism $(\pi,d)$ between $H \curvearrowright Y$ and $G \curvearrowright X$, 
     $$ (id_X,i)\circ (\pi, d)=(id_X \circ \pi, d'),$$
     where $d'(h,y)=i(d(h,y), \pi(y))=d(h,y)$. Thus, $(id_X,i)\circ (\pi, d)=(\pi,d)$.
\end{proof}

\begin{dfn}
    We define $\PDS$ the category whose objects are partial dynamical systems and morphisms are orbit morphisms.
\end{dfn}

As mentioned at the beginning of the section, we want to define a functor from $\PDS$ to $\TG$ that at the level of objects sends a partial dynamical system to its transformation groupoids. We need to describe what happens at the level of morphisms.

\begin{lem}\label{orbit mor. gives a mor.}
    Let $(\phi,a)$ be an orbit morphism between $G \curvearrowright X$ and $H \curvearrowright Y$, and define $\Theta_{(\phi,a)}:G \ltimes X  \to H \ltimes Y$ by $\Theta_{(\phi,a)}(g,x)=(a(g,x),\phi(x))$. Then $\Theta_{(\phi,a)}$ is a continuous homomorphism.
\end{lem}

\begin{proof}
    Define $\Theta:=\Theta_{(\phi,a)}:G \ltimes X  \to H \ltimes Y $ by $\Theta(g,x)=(a(g,x), \varphi(x))$ for all $g \in G$ and $x \in U_{g^{-1}}$.
Clearly, $\Theta$ is continuous. 
For $((g_1,x_1),(g_2,x_2)) \in (G \ltimes X)^{(2)}$, we see that
\begin{align*}  \Theta \times \Theta ((g_1, g_2.x_2),(g_2,x_2))&=(\Theta(g_1, g_2.x_2), \Theta(g_2,x_2)) \\
&=((a(g_1, g_2.x_2),\phi(g_2.x_2)),(a(g_2,x_2), \phi(x_2)) ).
\end{align*}
Since $\phi(g_2.x_2)=a(g_2,x_2).\phi(x_2)$, we have $s(a(g_1, g_2.x_2),\phi(g_2.x_2))=r(a(g_2,x_2), \phi(x_2))$. It means that $\Theta \times \Theta ((G \ltimes X)^{(2)}) \subseteq (H \ltimes Y)^{(2)}$.
Also, for $((g_1,x_1),(g_2,x_2)) \in (G \ltimes X)^{(2)}$, we have
\begin{align*} 
\Theta((g_1,x_1)(g_2,x_2)) &=\Theta(g_1g_2,x_2) \\
&=(a(g_1g_2,x_2), \phi(x_2)) \\
&=(a(g_1, g_2.x_2)a(g_2,x_2), \phi(x_2)) \\
&=(a(g_1, g_2.x_2), a(g_2,x_2).\phi(x_2))(a(g_2,x_2),\phi(x_2)) \\
&=(a(g_1, g_2.x_2), \phi(g_2.x_2))(a(g_2,x_2),\phi(x_2)) \\
&=(a(g_1, x_1), \phi(x_1))(a(g_2,x_2),\phi(x_2)) \\
&=\Theta(g_1,x_1)\Theta(g_2,x_2).
\end{align*}
Thus, $\Theta$ is a homomorphism. 
\end{proof}

\begin{thm}\label{thm: functor}
    There exists a fully faithful functor $F:\PDS\to\TG$ that associates a partial dynamical system $G\curvearrowright X$ with its transformation groupoid $G\ltimes X$, and an orbit morphism $(\phi,a)$ between $G \curvearrowright X$ and $H \curvearrowright Y$ with the continuous homomorphism $\Theta_{(\phi,a)}:G \ltimes X  \to H \ltimes Y$ given by $\Theta_{(\phi,a)}(g,x)=(a(g,x),\phi(x))$.
\end{thm}

\begin{proof}
    We start by proving that $F$ preserves compositions and identities.
    
    For a partial dynamical system $G \curvearrowright X$, let $(id_X,i)$ be its identity. Then,
    \begin{align*} \Theta_{(id_X,i)}(g,x)=(i(g,x), id_X(x))=(g,x)\end{align*}
Thus, $\Theta_{(id_X,i)}=id_{G \ltimes X}$.

 For $(\phi,a): G \curvearrowright X \to H \curvearrowright Y$ and $(\psi,b): H \curvearrowright Y \to K \curvearrowright Z$, 
\begin{align*} \Theta_{(\psi,b)\circ(\phi,a)}(g,x)=\Theta_{(\psi\circ\phi,c)}(g,x)=(c(g,x), \psi\circ\phi(x))= (b(a(g,x),\phi(x)), \psi\circ\phi(x)) ,
\end{align*}
and 
\begin{align*} \Theta_{(\psi,b)} \circ \Theta_{(\phi,a)}(g,x)=\Theta_{(\psi,b)}(a(g,x), \phi(x))=(b(a(g,x),\phi(x)), \psi\circ\phi(x)).
\end{align*}
Thus, $\Theta_{(\psi,b)\circ(\phi,a)}= \Theta_{(\psi,b)} \circ \Theta_{(\phi,a)}$.

Now, given partial dynamical systems $G \curvearrowright X$ and $H \curvearrowright Y$, we show that the function $F_{G \curvearrowright X,H \curvearrowright Y}:\Hom(G \curvearrowright X,H \curvearrowright Y)\to \Hom(G \ltimes X,H \ltimes Y)$ is bijective. We do this by building an inverse. Let
$\Theta:G \ltimes X  \to H \ltimes Y$ be a continuous homomorphism. Define $\phi_{\Theta}$ to be the restriction of $\Theta$ to the unit spaces, and $a_{\Theta}(g,x)$ to be the projection to the first coordinate of $\Theta(g,x)$. Note that $\Theta(g,x)=(a_{\Theta}(g,x),\phi_{\Theta}(x))$ by construction. Let us prove that $(\phi_{\Theta},a_{\Theta})$ is an orbit morphism.

\eqref{om preserves action} For $g\in G$ and $x\in U_{g^{-1}}$, we have that $\phi_{\Theta}(x)\in V_{a_{\Theta(g,x)^{-1}}}$ by construction. Moreover,
$$
    \phi_{\Theta}(g.x) =\Theta(g.x)=\Theta(r(g,x))=r(\Theta(g,x))=a_{\Theta}(g,x).\phi_{\Theta}(x).
$$

\eqref{om is a cocyle} Let $g_1, g_2\in G$ and $x\in  U_{g_2^{-1}}\cap U_{(g_1g_2)^{-1}}$. We have that
\begin{align*}
    (a_{\Theta}(g_1g_2,x),\phi_{\Theta}(x)) &=\Theta(g_1g_2,x)\\
    &=\Theta((g_1,g_2.x)(g_2,x))\\
    &=\Theta(g_1,g_2.x)\Theta(g_2,x)\\
    &=(a_{\Theta}(g_1,g_2.x),\phi_{\Theta}(g_2.x))(a_{\Theta}(g_2,x),\phi_{\Theta}(x))\\
    &=(a_{\Theta}(g_1,g_2.x),a_{\Theta}(g_2,x).\phi_{\Theta}(x))(a_{\Theta}(g_2,x),\phi_{\Theta}(x))\\
    &=(a_{\Theta}(g_1,g_2.x)a_{\Theta}(g_2,x),\phi_{\Theta}(x)).
\end{align*}
Thus $a_{\Theta}(g_1g_2,x)=a_{\Theta}(g_1,g_2.x)a_{\Theta}(g_2,x)$.

It is clear to see that the mapping $\Theta\mapsto (a_{\Theta},\phi_{\Theta})$ is then the inverse of $F_{G \curvearrowright X,H \curvearrowright Y}$.
\end{proof}

A question that arises from the above theorem is when an étale groupoid isomorphic to a transformation groupoid for some partial action on its unit space. The goal of next theorem is to answer this question.

\begin{thm}\label{thm:transf.groupoid}
	Let $\mathcal{G}$ be an étale groupoid with unit space $X$. The following are equivalent:
	\begin{enumerate}
		\item\label{i:iso} There exists a group $H$ and a topological partial action $\theta=(\{X_h\}_{h\in H},\{\theta_h:X_{h^{-1}}\to X_h\})$ of $H$ on $X$ such that $\mathcal{G}$ is isomorphic to $H\ltimes_{\theta} X$.
		\item\label{i:partition} There exists a partition $P$ of $\mathcal{G}$ such that the elements of $P$ are open bisections of $\mathcal{G}$ satisfying:
		\begin{itemize}
			\item $X\in P$,
			\item for all $U\in P$, we have $U^{-1}\in P$,
			\item for all $U,V\in P$, there exists $W\in P$ such that $UV\scj W$.
		\end{itemize}
		\item\label{i:cocyle} There exists a group $H$ and a continuous cocycle $c:\mathcal{G}\to H$ such that $c^{-1}(e)=X$, where $e$ is the identity of $H$.
	\end{enumerate}
\end{thm}

\begin{proof}
	\eqref{i:iso}$\Rightarrow$\eqref{i:partition}: It is enough to suppose that $\mathcal{G}=H\ltimes_{\theta} X$. For the partition $P=\{\{h\}\times X_{h^{-1}}\}$, it is straightforward to check that $P$ satisfies the conditions in \eqref{i:partition}.
	
	\eqref{i:partition}$\Rightarrow$\eqref{i:cocyle}: Let $P$ be as in \eqref{i:partition} and consider the group $H$ generated by $\{[U]:U\in P\}$ subject to the relations:
	\begin{itemize}
		\item $[X]$ is the identity of $H$,
		\item $[U]^{-1}=[U^{-1}]$,
		\item $[U][V]=[W]$ if $UV$ is non-empty and $UV\scj W$.
	\end{itemize}
	We note that there is a unique $W$ in the last item since $P$ is a partition. For each $\gamma\in\mathcal{G}$, there is a unique $U_\gamma\in P$ such that $\gamma\in U_{\gamma}$. We can define $c:\mathcal{G}\to H$ by $c(\gamma)=[U_{\gamma}]$. For $\gamma\in\mathcal{G}$, we have that $[U_{\gamma^{-1}}]=[U_{\gamma}^{-1}]=[U_\gamma]^{-1}$. Hence $c(\gamma^{-1})=c(\gamma)^{-1}$. And for $(\gamma,\eta)\in\mathcal{G}^{(2)}$, since $\gamma\eta\in U_{\gamma\eta}$, we have that $[U_\gamma][U_\eta]=[U_{\gamma\eta}]$ so that $c(\gamma\eta)=c(\gamma)c(\eta)$. Thus, $c$ is a coclye.
	
	We now check that $c$ is continuous. Using the relations, we see that each element of $H$ is a product $[U_1]\cdots[U_n]$ for some $U_1,\ldots,U_n\in P$, and if we cannot reduce to an element $[W]$, then $U_1\cdots U_n$ is empty as a product of bisections. It follows that
	\[c^{-1}(h)=\begin{cases}
		U, & \text{if }h=[U]\text{ for some }U\in P\\
		\emptyset, &\text{otherwise.}
	\end{cases}\]
	It follows that $c$ is a continuous cocycle such that $c^{-1}(e)=X$.
	
	\eqref{i:cocyle}$\Rightarrow$\eqref{i:iso}: Note that for each $h\in H$, $c^{-1}(h)$ is an open bisection. Indeed, if $\gamma,\eta\in c^{-1}(h)$ are such that $s(\gamma)=s(\eta)$, then $(\gamma,\eta^{-1})\in\mathcal{G}^{(2)}$, and $c(\gamma\eta^{-1})=hh^{-1}=e$. By hypothesis, $\gamma\eta^{-1}\in X$, from where it follows that
	\[\gamma\eta^{-1}=r(\gamma\eta^{-1})=r(\gamma)=\gamma\gamma^{-1},\]
	and thus $\gamma=\eta$. Similarly, we prove that if $\gamma,\eta\in c^{-1}(h)$ and $r(\gamma)=r(\eta)$, then $\gamma=\eta$. Since $\mathcal{G}$ is étale, for each $h\in H$, the set  $X_h=s(c^{-1}(h^{-1}))$ is an open subset of $X$ and the map $\theta_h:X_{h^-1}\to X_h$ given by $\theta(x)=r((s|_{c^{-1}(h)})^{-1}(x))$  is a homeomorphism. 
 	Since $c^{-1}(e)=X$, we have that $X_e=X$ and $\theta_e(x)=x$ for all $x\in X$. Now, for $x\in X$ and $g,h\in H$ such that $x\in X_{h^{-1}}$ and $\theta_h(x)\in X_{g^{-1}}$, there are $\gamma\in c^{-1}(g)$ and $\eta\in c^{-1}(h)$ such that $s(\eta)=x$ and $r(\eta)=\theta_h(x)=s(\gamma)$. Hence $\theta_g(\theta_h(x))=r(\gamma)$. On the other hand, since $s(\gamma\eta)=s(\eta)=x$, we have that $x\in X_{(gh)^{-1}}$ and $\theta_{gh}(x)=r(\gamma\eta)=r(\gamma)=\theta_g(\theta_h(x))$. If follows that $\theta_g\circ\theta_h\leq \theta_{gh}$. Thus, we obtain a partial action $\theta=(\{X_h\}_{h\in H},\{\theta_h:X_{h^{-1}}\to X_h\})$ of $H$ on $X$.
	
	By construction and the above computation, we see that the map $\Phi:\mathcal{G}\to H\ltimes_{\theta} X$ given by $\Phi(\gamma)=(c(\gamma),s(\gamma))$ is a continuous bijective homomorphism with inverse given by $\Phi^{-1}(h,x)=(s|_{c^{-1}(h)})^{-1}(x)$. To see that $\Phi$ is open, note that the family of open bisections $U$ such that $U\scj c^{-1}(h)$ for some $h\in H$ forms a basis for the topology on $\mathcal{G}$. For such bisection $\Phi(U)=\{h\}\times s(U)$, which is open in $H\ltimes_{\theta} X$. Thus, $\Phi$ is an isomorphism of topological groupoids.
\end{proof}

\begin{cor}
    Every discrete groupoid is isomorphic to the transformation groupoid of some partial dynamical system.
\end{cor}

\begin{proof}
    Let $\mathcal{G}$ be a discrete groupoid with unit space $X$ and consider the partition of $\mathcal{G}$ consisting of $X$ and $\{\gamma\}$ for every $\gamma\in \mathcal{G}\setminus X$. Because the groupoid is discrete, the partition $P$ satisfies Theorem \ref{thm:transf.groupoid}(2), from where the result follows.
\end{proof}

Next, we will give a characterization of isomorphisms between transformation groupoids of partial dynamical systems. First, we need a lemma.

\begin{lem}\label{isom pds}
    Let $(\phi,a)$ be an orbit morphism between $G\curvearrowright X$ and $H\curvearrowright Y$. We have that $(\phi,a)$ is an isomorphism in $\PDS$ if and only if $\phi$ is a homeomorphism and for each $x\in X$, the map $a(\cdot,x):G_x\to H_{\phi(x)}$ is a bijection.
\end{lem}

\begin{proof}
    Suppose first that $(\phi,a)$ is an isomorphism with inverse $(\psi,b)$. By the definition of composition between orbit morphisms, we see that $\psi=\phi^{-1}$ and that $b(\cdot,\phi(x))$ is the inverse of $a(\cdot,x)$ for each $x\in X$.

    Now, suppose that $\phi$ is a homeomorphism and that $a(\cdot,x):G_x\to H_{\phi(x)}$ is bijective for each $x\in X$. We define the map $b:\bigcup_{h\in H}\{h\}\times V_{h^{-1}}$ by $b(h,y)=a(\cdot,\phi^{-1}(y))^{-1}(h)$. In other words, $b(h,y)$ is the unique element in $G_{\phi^{-1}(y)}$ such that $a(b(h,y),\phi^{-1}(y))=h$. The next goal is to prove that $(\phi^{-1},b)$ is an inverse of $(\phi,a)$ in $\PDS$. We divide in a few steps.

    We begin by showing that $b$ is continuous, which is equivalent to $b$ being locally constant. Fix $(h,y)\in \bigcup_{h\in H}\{h\}\times V_{h^{-1}}$. Since $a$ is continuous, there exists an open neighbourhood $U\scj U_{b(h,y)^{-1}}$ of $\phi^{-1}(y)$ such that $a(b(h,y),x)=h$ for all $x\in U$. Let $V=\phi(U)$, which is an open neighbourhood of $y$ since $\phi$ is a homeomorphism. By \eqref{om preserves action} of Definition~\ref{def orbit morphism}, we have that $V\scj V_{h^{-1}}$. Moreover, for $y'\in V$, we have that $b(h,y')=h$ since $\phi^{-1}(y')\in U$. Hence, $b$ is constant in $\{h\}\times V$. It follows that $b$ is continuous.

    Now, we show that $(\phi^{-1},b)$ satisfies condition \eqref{om preserves action} of Definition~\ref{def orbit morphism}. Let $(h,y)\in \bigcup_{h\in H}\{h\}\times V_{h^{-1}}$. By construction, $\phi^{-1}(y)\in U_{b(h,y)^{-1}}$. Moreover, because $(\phi,a)$ is an orbit morphism, we obtain
    \[\phi(b(h,y).\phi^{-1}(y))=a(b(h,y),\phi^{-1}(y)).\phi(\phi^{-1}(y))=h.y.\]
    Hence $\phi^{-1}(h.y)=b(h,y).\phi^{-1}(y)$.

    Next, we show that $(\phi^{-1},b)$ satisfies condition \eqref{om is a cocyle} of Definition~\ref{def orbit morphism}. Let $h_1,h_2\in H$ and $y\in V_{h_2^{-1}}\cap V_{(h_1h_2)^{-1}}$. Consider $g:=b(h_1h_2,y)$, $g_1:=b(h_1,h_2.y)$ and $g_2:=b(h_2,y)$. We need to prove that $g=g_1g_2$. By the construction of $b$, we have that
    \begin{align*}
        a(g,\phi^{-1}(y))&=h_1h_2\\
        &=a(g_1,\phi^{-1}(h_2.y))a(g_2,\phi^{-1}(y))\\
        &=a(g_1,b(h_2,y).\phi^{-1}(y))a(g_2,\phi^{-1}(y))\\
        &=a(g_1,g_2.\phi^{-1}(y))a(g_2,\phi^{-1}(y))\\
        &=a(g_1g_2,\phi^{-1}(y)).
    \end{align*}
    Since $a(\cdot,\phi^{-1})$ is a bijection, we have that $g=g_1g_2$.

    Finally, we prove that $(\phi^{-1},b)$ is the inverse of $(\phi,a)$ in $\PDS$. By construction $(\phi,a)\circ(\phi^{-1},b)=id_{H\curvearrowright Y}$. On the other hand, for $g\in G$ and $x\in U_{g^{-1}}$, $b(a(g,x),\phi(x))=a(\cdot,x)^{-1}(a(g,x))=g$. Hence $(\phi^{-1},b)\circ(\phi,a)=id_{G\curvearrowright X}$.
\end{proof}

 Given partial dynamical systems $G \curvearrowright X$ and $H \curvearrowright Y$, we recall that $G \curvearrowright X$ and $H \curvearrowright Y$ are {\it continuously orbit equivalent} (\cite[Definition 2.6]{Li2017}) if there exist a homeomorphism $\phi: X \to Y$ and  continuous maps $a: \bigcup_{g \in G} \{g\}\times U_{g^{-1}} \to H$, $b: \bigcup_{h \in H} \{h\} \times V_{h^{-1}} \to G$  such that 
\begin{align} 
\phi(g.x)&=a(g,x).\phi(x), \\
\phi^{-1}(h.y)&=b(h,y). \phi^{-1}(y).
\end{align}
We call $(\phi,a,b)$  a continuous orbit equivalence.

\begin{dfn} Let   $G \curvearrowright X$ and $H \curvearrowright Y$ be partial dynamical systems. Consider a map $a: \bigcup_{g \in G} \{g\}\times U_{g^{-1}} \to H$ and let   $\phi: X \to Y$ be a homeomorphism.
 We say that $(\phi,a)$ {\it preserves stabilisers} if 
$a(\cdot, x)$ restricts to bijections $ \operatorname{Stab}(x) \to \operatorname{Stab}(\phi(x)) $, and that $(\phi,a)$ {\it preserves essential stabilisers} if $a(\cdot, x)$ restricts to bijections $ \operatorname{Stab}^{ess}(x) \to \operatorname{Stab}^{ess}(\phi(x)) $.
\end{dfn}

We remark that $(\phi,a)$ above does not need to be an orbit morphism.

\begin{thm}\label{isom-equivalences}
    Let $G \curvearrowright X$ and $H \curvearrowright Y$ be partial dynamical systems. The following are equivalent:
    \begin{enumerate}
        \item $G \curvearrowright X$ and $H \curvearrowright Y$ are isomorphic in $\PDS$.
        \item $G \ltimes X$ and $H\ltimes Y$ are isomorphic as topological groupoids.
        \item There exists a continuous orbit equivalence $(\phi,a,b)$ such that $a$ is a cocycle and $(\phi,a)$ preserves stabilisers.
        \item There exists a continuous orbit equivalence $(\phi,a,b)$ such that $a$ is a cocycle and $(\phi,a)$ preserves essential stabilisers.
    \end{enumerate}
\end{thm}

\begin{proof} (1)$\iff$(2): The result follows from the fact that the functor $F$ of Theorem~\ref{thm: functor} is fully faithful.

    (1)$\implies$(3): Let $(\phi,a)$ be an isomorphism between $G \curvearrowright X$ and $H \curvearrowright Y$ with inverse given by $(\phi^{-1},b)$. Clearly, $(\phi,a,b)$ is a continuous orbit equivalence such that $a$ is cocycle.  To prove that $(\phi,a)$ preserves stabilisers, fix $x \in X$. If $a(g_1,x)=a(g_2,x)$ for $g_1, g_2 \in \operatorname{Stab}(x)$, then $$g_1=i(g_1,x)=b(a(g_1,x), \phi(x))=b(a(g_2,x), \phi(x))=i(g_2,x)=g_2$$
since $(\phi^{-1},b) \circ (\phi,a)=(id_X, i)$. Thus, $a(\cdot, x)|_{\operatorname{Stab}(x)}$ is injective.

To show that $a(\cdot, x)|_{\operatorname{Stab}(x)}$ is surjective,
choose $h \in \operatorname{Stab}(\phi(x))$. Put $y=\phi(x)$. Then, 
$$b(h, y).x=b(h, y).\phi^{-1}(y)=\phi^{-1}(h.y)=\phi^{-1}(h.\phi(x))=\phi^{-1}(\phi(x))=x.$$ Thus, $b(h,y) \in \operatorname{Stab}(x)$.
Since $(\phi,a) \circ (\phi^{-1},b) =(id_Y,i)$, it follows that 
$$a(b(h,y),x)=a(b(h,y), \phi^{-1}(y))=i(h,y)=h.$$
 Hence, 
$a(\cdot, x)|_{\operatorname{Stab}(x)}$ is surjective. 

(3)$\implies$(4): It is enough to show that $g \in \operatorname{Stab}^{ess}(x)$ if and only if $a(g,x) \in \operatorname{Stab}^{ess}(\varphi(x))$. First, let $g \in \operatorname{Stab}^{ess}(x)$. Then, there is an open neighborhood $U$ of $x$ such that $g \in \operatorname{Stab}(x')$ for all $x'\in U$. So, $U \subseteq U_{g^{-1}}$ and $g.x'=x'$ for all $x'\in U$. Since $a$ is continuous and $G$ is discrete, one may assume that $a(g,x)=a(g,x')$ for all $x' \in U$. Then, $U':=\phi(U)$ is an open neighborhood of $\phi(x)$ such that for $y \in U'$,
$$ a(g,x).y=a(g,\phi^{-1}(y)).\phi(\phi^{-1}(y))=\phi(g.\phi^{-1}(y))=\phi(\phi^{-1}(y))=y.$$
Thus, $a(g,x)  \in \operatorname{Stab}^{ess}(\phi(x))$.

Now, suppose that $a(g,x) \in \operatorname{Stab}^{ess}(\phi(x))$. Then, there is an open neighborhood $U'$ of $\phi(x)$ such that $a(g,x).y=y$ for all $y \in U'$. Since $a$ and $\phi$ are continuous, there is an open neighborhood $U$ of $x$ such that $\phi(U) \subseteq U'$ and $a(g,x')=a(g,x)$ for all $x' \in U$. Thus, for all $x' \in U$, 
$$\phi(g.x')=a(g,x').\phi(x') =a(g,x).\phi(x')=\phi(x'),$$
and hence, $g.x'=x'$ for all $x' \in U$. Hence, $g \in \operatorname{Stab}^{ess}(x)$.

(4)$\implies$(1): We show that $(\phi,a)$ is an isomorphism between $G \curvearrowright X$ and $H \curvearrowright Y$ (although the inverse is not necessarily $(\phi^{-1},b)$). By Lemma~\ref{isom pds}, it is sufficient to show that for each $x\in X$, the map $a(\cdot,x):G_x\to H_{\phi(x)}$ is a bijection.

Let $g_1,g_2\in G$ and $x\in U_{g_1^{-1}}\cap U_{g_2^{-1}}$ be such that $a(g_1,x)=a(g_2,x)=:h$. Since $a$ is continuous, there exists an open neighborhood $U\scj U_{g_1^{-1}}\cap U_{g_2^{-1}}$ of $x$ such that for all $x'\in U$, we have that $a(g_1,x')=h=a(g_2,x')$. In particular
\[\phi(g_1.x')=a(g_1,x').\phi(x')=a(g_2,x').\phi(x')=\phi(g_2.x'),\]
from where it follows that $g_2^{-1}g_1\in \operatorname{Stab}^{ess}(x)$. Since $a$ is a cocyle, we then have
\begin{align*}a(g_2^{-1}g_1, x)&= a(g_2^{-1}, g_1.x)a(g_1,x)\\
&=a(g_2^{-1}, g_2.x)a(g_1,x)\\
&=a(g_2^{-1}, g_2.x)a(g_2,x)\\
&=a(g_2^{-1} g_2, x)\\
&=a(e_G, x).
\end{align*}
By hypothesis, $a(\cdot,x)$ restricts to a bijection between $\operatorname{Stab}^{ess}(x)$ and $\operatorname{Stab}^{ess}(\phi(x))$. It then follows that $g_2^{-1}g_1=e_G$, that is, $g_1=g_2$.

Now let $x\in X$ and $h\in H_{\phi(x)}$. Because $a$ and $b$ are continuous, and $\phi$ is a homeomorphism, we can find an open neighborhood $U$ of $x$ such that $U\scj U_{b(h,\phi(x))^{-1}}$, $\phi(U)\scj V_{h^{-1}}$ and $a(b(h,\phi(x')),x')=a(b(h,\phi(x)),x)$ for all $x'\in U$. Then, for $x'\in U$, we have that
\begin{align*}a(b(h,\phi(x)), x).\phi(x')&=a(b(h,\phi(x')), x').\phi(x')\\
&=\phi(b(h,\phi(x')).x') \\
&=\phi(b(h,\phi(x')).\phi^{-1}(\phi(x'))) \\
&=\phi(\phi^{-1}(h.\phi(x'))) \\
&=h.\phi(x').
\end{align*}
Thus, $a(b(h,\phi(x)), x)^{-1}h\in \operatorname{Stab}^{ess}(\phi(x))$. By hypothesis, there exists $g'\in \operatorname{Stab}^{ess}(x)$ such that $a(g',x)=a(b(h,\phi(x)), x)^{-1}h$. For $g=b(h,\phi(x))g'$, we then have that $x\in U_{g^{-1}}$ and
\begin{align*}
    a(g,x)&=a(b(h,\phi(x))g',x)\\
    &=a(b(h,\phi(x)),g'.x)a(g',x)\\
    &=a(b(h,\phi(x)),x)a(b(h,\phi(x)), x)^{-1}h\\
    &=h.
\end{align*}
Thus $a(\cdot,x)$ is a bijection. It then follows from Lemma~\ref{isom pds} that $(\phi,a)$ is an isomorphism in $\PDS$.
\end{proof}

For topologically free partial dynamical systems, the essential stabalizers groups are all trivial. Also, if $(\phi,a,b)$ is a continuous orbit equivalence between $G \curvearrowright X$ and $H \curvearrowright Y$, then $a$ is a cocyle by a simple generalization of \cite[Lemma~2.8]{Li2018}. The following it then a generalization of \cite[Theorem 1.2]{Li2017} and \cite[Corollary~7.5]{CRST}.

\begin{prop}\label{prop:isom-equivalences}
    Let $G \curvearrowright X$ and $H \curvearrowright Y$ be partial dynamical systems.  Suppose that $\operatorname{Stab}^{ess}(x)$ and  $\operatorname{Stab}^{ess}(y)$ are torsion-free and abelian for all $x \in X$ and all $y \in Y$. Then, each one of the items (1)-(4) of Theorem~\ref{isom-equivalences} is also equivalent to:
    \begin{enumerate}
        \setcounter{enumi}{4}
        \item There is an isomorphism $\Phi: C_0(X)\rtimes_r G \to C_0(Y)\rtimes_r H $ such that $\Phi(C_0(X))=C_0(Y)$.
    \end{enumerate}
\end{prop}

\begin{proof}
    Consider  $G \ltimes X $ and $ H \ltimes Y$ equipped with trivial grading. Then, the equivalence between (2) and (5) follows from Lemma \ref{partial system:torsion free and abelian} and \cite[Theorem 6.2]{CRST}.
\end{proof}

\section{Partial actions of the free group}\label{Sec 4}

Let $\mathbb{F}$ be a free group generated by a set $\CA$ and, for each $g\in\mathbb{F}$, denote by $|g|$ the length of the reduced form of $g$. Following \cite[Definition~4.9]{ExelBook} we say that a partial dynamical system $\mathbb{F}\curvearrowright X$ is \emph{semi-saturated} if $\varphi_{gh}=\varphi_g\circ\varphi_h$ for all $g,h\in\mathbb{F}$ such that $|g+h|=|g|+|h|$. By \cite[Proposition~4.10]{ExelBook}, a semi-saturated partial dynamical system is completely defined by the homeomorphisms $\{\varphi_\alpha:U_{\alpha^{-1}}\to U_\alpha\}_{\alpha\in\CA}$. Note that, in this case, for $\alpha,\beta\in\mathbb{F}_+$, we have that $U_{\alpha\beta}\scj U_{\alpha}$. Moreover, if $g\in\mathbb{F}$ cannot be written as $g=\alpha\beta^{-1}$ for $\alpha,\beta\in\mathbb{F}_+$, then $U_g=\emptyset$.

We say that a semi-saturated partial dynamical system is \emph{orthogonal} if $U_\alpha\cap U_\beta=\emptyset$ for every $\alpha,\beta\in\CA$ such that $\af\neq \beta$. In this case, we can define a Deaconu-Renault system by $U:=\bigsqcup_{\alpha\in\CA}U_\alpha$ and $\sigma:U\to X$ is given by $\sigma(x)=\phi_{\alpha^{-1}}(x)$ if $x\in U_\alpha$. Note that for every $n\in\mathbb{N}\setminus\{0\}$, $\operatorname{dom}(\sigma^n)=\bigsqcup_{\alpha\in\mathbb{F}_+,|\alpha|=n}U_{\alpha}$ and $\sigma^n(x)=\varphi_{\alpha^{-1}}$ if $x\in U_\alpha$.
Here, we use both $(X, \sigma)$ and $\sigma: U \to X$ interchangeably to represent the Deaconu-Renault system.

The following theorem generalizes \cite[Theorem 5.5]{CastroWyk} and \cite[Theorem 4.4]{CasK2}.

\begin{thm}\label{thm: from PA to DRS}
    Let $\mathbb{F}\curvearrowright X$ be an orthogonal semi-saturated partial dynamical system, and $\sigma: U \to X$ the corresponding Deaconu-Renault system. There is an isomorphism of topological groupoids $\Xi:\mathbb{F}\ltimes X\to \Gamma(X,\sigma)$ such that $\Xi(\alpha\beta^{-1},x)=(\varphi_{\alpha\beta^{-1}}(x),|\alpha|-|\beta|,x)$ for every $\alpha,\beta\in\mathbb{F}_+$ such that $\alpha\beta^{-1}$ is in reduced form and $x\in U_{\beta\alpha^{-1}}$.
\end{thm}

\begin{proof}
    Note that $\Xi$ is completely determined by the expression in the statement, since $U_g=\emptyset$ for $g\in\mathbb{F}$ such that we cannot write $g$ as $\alpha\beta^{-1}$ for $\alpha,\beta\in\mathbb{F}_+$. Moreover, if $g=\alpha\beta^{-1}$ in reduced form, then $\varphi_{g}=\varphi_{\alpha}\circ\varphi_{\beta^{-1}}$. Hence, for $x\in U_{\beta\alpha^{-1}}$, we have that $x\in U_{\beta}\scj\operatorname{dom}(\sigma^{|\beta|})$ and $\varphi_{\alpha\beta^{-1}}(x)\in U_\alpha\scj\operatorname{dom}(\sigma^{|\alpha|})$. Also
    $$\sigma^{|\alpha|}(\varphi_{\alpha\beta^{-1}}(x))=\varphi_{\alpha^{-1}}(\varphi_{\alpha}(\varphi_{\beta^{-1}}(x)))=\varphi_{\beta^{-1}}(x)=\sigma^{|\beta|}(x),$$
    from where it follows that $\Xi$ is well-defined.

    We now find an inverse for $\Xi$. Let $(x,k,y)\in\Gamma(X,\sigma)$, so that there exist $n,m\in\mathbb{N}$ such that $x\in\operatorname{dom}(\sigma^n)$, $y\in\operatorname{dom}(\sigma^m)$, $k=n-m$ and $\sigma^n(x)=\sigma^m(y)$. Then, there are $\alpha,\beta\in\mathbb{F}_+$ such that $|\alpha|=n$, $|\beta|=m$, $x\in U_{\alpha}$ and $y\in U_{\beta}$. Suppose that the reduced form of $\alpha\beta^{-1}$ is $\alpha'\beta'^{-1}$, so that there exists $\delta\in\mathbb{F}_+$ such that $\alpha=\alpha'\delta$ and $\beta=\beta'\delta$. Hence, $x\in U_{\alpha'}$, $y\in U_{\beta'}$, $|\alpha|-|\beta|=|\alpha'|-|\beta'|$ and
    $$\sigma^{|\alpha'|}(x)=\varphi_{\alpha'^{-1}}(x)=\varphi_{\delta}(\varphi_{\alpha^{-1}}(x))=\varphi_{\delta}(\varphi_{\beta^{-1}}(y))=\varphi_{\beta'^{-1}}(y)=\sigma^{|\beta'|}(y).$$
    The inverse of $\Xi$ then sends $(x,k,y)$ to $(\alpha'\beta'^{-1},y)$.

    The proof that $\Xi$ is a homomorphism of groupoids follows the same steps as the particular cases, and we omit it here (see, for instance, \cite[Theorem~4.4]{CasK2}).

    It remains to prove that $\Xi$ is a homeomorphism. Given $\alpha,\beta\in\mathbb{F}_+$ such that $\alpha\beta^{-1}$ is in reduced form and $U\scj U_{\beta\alpha^{-1}}$ is open, we have that $\Xi(\{\alpha\beta^{-1}\}\times U)=Z(\varphi_{\alpha\beta^{-1}}(U),|\alpha|,|\beta|,U)$ so that $\Xi$ is open. On the other hand, let $n,m\in\mathbb{N}$, $U\scj\operatorname{dom}(\sigma^n)$ and $V\scj\operatorname{dom}(\sigma^m)$ be open. Then $$\Xi^{-1}(Z(U,n,m,V))=\bigsqcup_{\alpha,\beta\in\mathbb{F}_+,|\alpha|=n,|\beta|=m}\{\alpha\beta^{-1}\}\times (V\cap U_{\beta\alpha^{-1}}),$$
    from where it follows that $\Xi$ is continuous.
\end{proof}

\begin{lem}\label{char:isotropy of DR groupoid}Let $\mathbb{F}\curvearrowright X$ be an orthogonal semi-saturated partial dynamical system, and $\sigma:U\to X$ the corresponding Deaconu-Renault system. If $(x,k,x) \in \operatorname{Iso}(\Gamma(X,\sigma)) $, then there exist $\delta, \gamma \in \mathbb{F}_+$ such that either $k=|\delta|$ and $x=\varphi_{\dt\gm\dt^{-1}}(x)$ or $k=-|\delta|$ and $x=\varphi_{\dt\gm^{-1}\dt^{-1}}(x)$.
\end{lem}

\begin{proof} If $(x,k,x) \in \operatorname{Iso}(\Gamma(X,\sigma)) $, then there exist $m,n\in\mathbb{N}$ such that $x\in\operatorname{dom}(\sigma^m) \cap \operatorname{dom}(\sigma^n)$, $k=m-n$ and $\sigma^m(x)=\sigma^n(x)$. Then, there are $\alpha,\beta\in\mathbb{F}_+$ such that $|\alpha|=m$, $|\beta|=n$, $x\in U_{\alpha}\cap U_{\beta} (\neq \emptyset)$.  
Then, since $\mathbb{F}\curvearrowright X$ is orthogonal, we have either $\alpha=\beta\gamma$ or $\beta=\alpha\gamma$ for some $\gamma \in \mathbb{F}_+$.
Since $\mathbb{F}\curvearrowright X$ is semi-saturated, we have that
$$x=\varphi_{\alpha\bt^{-1}}(x)= 
\left\{
   \begin{array}{ll}
    \varphi_{\beta\gamma\beta^{-1}}(x)  & \hbox{if\ }  \alpha=\beta\gamma,   \\
\varphi_{\alpha\gamma^{-1}\alpha^{-1}}(x) &  \hbox{if\ }  \beta=\alpha\gamma.
        \end{array}
\right. $$
Thus, one concludes that there exist $\delta (=\beta ~\text{or}~ \alpha), \gamma \in \mathbb{F}_+$ such that $k=\pm|\delta|$ and $x=\varphi_{\dt\gm\dt^{-1}}(x)$ or $x=\varphi_{\dt\gm^{-1}\dt^{-1}}(x)$.
\end{proof}

\begin{prop}\label{PA:torsion free and abelian} For an orthogonal semi-saturated partial dynamical system $\mathbb{F}\curvearrowright X$, $\operatorname{Iso}(\mathbb{F}\ltimes X)$ is torsion-free  and abelian.
\end{prop}

\begin{proof} We first  prove that  the isomorphism $\Xi:\mathbb{F}\ltimes X\to \Gamma(X,\sigma)$ given by  $\Xi(\alpha\beta^{-1},x)=(\varphi_{\alpha\beta^{-1}}(x),|\alpha|-|\beta|,x)$   restricts to an isomorphism between $\operatorname{Iso}(\mathbb{F}\ltimes X)$  and $\operatorname{Iso}(\Gamma(X,\sigma))$. To do that, we only need to show that $\Xi|_{\operatorname{Iso}(\mathbb{F}\ltimes X)}: \operatorname{Iso}(\mathbb{F}\ltimes X) \to \operatorname{Iso}(\Gamma(X,\sigma))$ is onto. Choose $(x,k,x) \in \operatorname{Iso}(\Gamma(X,\sigma)) $. 
Then, by Lemma \ref{char:isotropy of DR groupoid}, there exist $\delta, \gamma \in \mathbb{F}_+$ such that either $k=|\delta|$ and $x=\varphi_{\dt\gm\dt^{-1}}(x)$ or $k=-|\delta|$ and $x=\varphi_{\dt\gm^{-1}\dt^{-1}}(x)$. One also sees that either  $\Xi(\dt\gm\dt^{-1}, x)=(x,k,x)$ or  $\Xi(\dt\gm^{-1}\dt^{-1}, x)=(x,k,x)$. It is clear that $(\dt\gm\dt^{-1}, x), (\dt\gm^{-1}\dt^{-1}, x) \in \operatorname{Iso}(\mathbb{F}\ltimes X)$. Thus,  $\Xi|_{\operatorname{Iso}(\mathbb{F}\ltimes X)}: \operatorname{Iso}(\mathbb{F}\ltimes X) \to \operatorname{Iso}(\Gamma(X,\sigma))$ is onto.

Now, since $\operatorname{Iso}(\Gamma(X,\sigma))$ is torsion-free and abelian, we have that $\operatorname{Iso}(\mathbb{F}\ltimes X)$ is torsion-free  and abelian.
\end{proof}

\begin{cor}\label{isom-equivalences:oss} Let $\mathbb{F}\curvearrowright X$ and $\mathbb{F}'\curvearrowright Y$ be orthogonal semi-saturated partial dynamical systems  and  $(X,\sigma)$, $(Y,\tau)$ the corresponding Deaconu-Renault systems.  Then the following are equivalent:
\begin{enumerate}
\item $\mathbb{F}\curvearrowright X$ and $\mathbb{F}'\curvearrowright Y$  are isomorphic in $\PDS$.
\item $ \mathbb{F} \ltimes X$ and $ \mathbb{F}' \ltimes Y$ are isomorphic as topological groupoids.
\item There exists a continuous orbit equivalence $(\phi,a,b)$ between $\mathbb{F}\curvearrowright X$ and $\mathbb{F}'\curvearrowright Y$  such that $a$ is a cocycle and $(\phi,a)$ preserves essential stabilisers.
\item There is an isomorphism $\Phi: C_0(X)\rtimes \mathbb{F} \to C_0(Y)\rtimes \mathbb{F}' $ such that $\Phi(C_0(X))=C_0(Y)$.
\item There is an essential-stabiliser-preserving continuous orbit equivalence $(\phi, k,l,k',l')$  from $(X, \sigma)$ to $(Y, \tau)$.
\item $\Gamma(X, \sigma)$ and $ \Gamma(Y, \tau)$ are isomorphic as topological groupoids.
\item There is an isomorphism $\Phi':C^*( \Gamma(X, \sigma)) \to C^*(\Gamma(Y, \tau))$ such that $\Phi'(C_0(X))=C_0(Y)$.
\end{enumerate}
\end{cor}

\begin{proof} (1) $\iff$ (2) $\iff$ (3) $\iff$ (4) follow from Proposition \ref{prop:isom-equivalences} and Proposition \ref{PA:torsion free and abelian}.

(2)$\iff$(6) follows form Theorem \ref{thm: from PA to DRS}.

(5) $\iff$ (6) $\iff$ (7) follow from \cite[Theorem 8.2]{CRST}.
\end{proof}

The equivalence between (3) and (5) in Corollary \ref{isom-equivalences:oss} follows from the isomorphism between the transformation groupoid and the Deaconu-Renault groupoid. And for that, we need the hypothesis that the continuous orbit equivalences preserve stabilisers. Next, we show that we can remove this hypothesis and still obtain an equivalence between the notions of orbit equivalence in the context of this section.

\begin{prop}\label{coe-coe}
Let $\mathbb{F}\curvearrowright X$ and $\mathbb{F}'\curvearrowright Y$ be orthogonal semi-saturated partial dynamical systems and $\sigma:U\to X$, $\tau:V\to Y$ the corresponding Deaconu-Renault systems. Then, $\mathbb{F}\curvearrowright X$ and $\mathbb{F}'\curvearrowright Y$ are continuous orbit equivalent if and only if $(X,\sigma)$ and $(Y,\tau)$ are continuous orbit equivalent.
\end{prop}

\begin{proof}
    First, let $(\phi,a,b)$ is a continuous orbit equivalence between $\mathbb{F}\curvearrowright X$ and $\mathbb{F}'\curvearrowright Y$. By hypothesis, for every $g\in\mathbb{F}$ and $x\in U_{g^{-1}}$, we have that $\phi(x)\in V_{a(g,x)^{-1}}$. This implies that $a(g,x)=\mu\nu^{-1}$ for $\mu,\nu\in\mathbb{F}'_+$. Now, suppose that $x\in U$, that is $x\in U_{\alpha}$ for some $\alpha\in\CA$ and define $k(x)=|\mu|$ and $l(x)=|\nu|$, where $a(\alpha^{-1},x)=\mu\nu^{-1}$, which is in reduced form in $\mathbb{F}'$. Using that $(\phi,a,b)$ is a continuous orbit equivalence and the partial action is semi-saturated, we obtain $\phi(\alpha^{-1}.x)=(\mu\nu^{-1}).\phi(x)=\mu.(\nu^{-1}.\phi(x))$, from where it follows that
    \[\tau^{l(x)}(\phi(x))=\nu^{-1}.\phi(x)=\mu^{-1}.(\mu.(\nu^{-1}.\phi(x)))=\mu^{-1}.\phi(\alpha^{-1}.x)=\tau^{k(x)}(\phi(\sigma(x))).\]
    This gives us the desired maps $k:U\to\mathbb{N}$ and $l:V\to\mathbb{N}$. Because $a$ is continuous, for $x\in U_{\alpha}$, we have that $a$ is constant in a neighborhood of $(\alpha^{-1},x)$, and hence $k$ and $l$ are constant in a neighborhood of $x$. This proves the continuity of $k$ and $l$. The maps $k'$ and $l'$ are built similarly from $\phi^{-1}$ and $b$.

    Suppose now that $(\phi,k,l,k',l')$ is a continuous orbit equivalence between $(X,\sigma)$ and $(Y,\tau)$. For any $x\in X$, we set $a(e_{\mathbb{F}},x)=e_{\mathbb{F}'}$. Now, let $\alpha\in\CA$ and $x\in U_{\alpha}$. Since $\tau^{l(x)}(\phi(x))=\tau^{k(x)}(\phi(\sigma(x)))$, we have that there are unique $\mu,\nu\in\mathbb{F}'_+$ such that $|\mu|=k(x)$, $|\nu|=l(x)$, $\phi(x)\in V_{\nu}$, $\phi(\sigma(x))\in V_{\mu}$, $\tau^{l(x)}(\phi(x))=\nu^{-1}.\phi(x)$ and $\tau^{k(x)}(\phi(\sigma(x)))=\mu^{-1}.\sigma(x)$. We set $a(\alpha^{-1},x)=\mu\nu^{-1}$. Let us prove that the map $a(\alpha^{-1},\cdot):U_{\alpha}\to \mathbb{F}'$ is continuous, $\phi(x)\in V_{a(\alpha^{-1},x)^{-1}}$ and $a(\alpha^{-1},x).\phi(x)=\phi(\alpha^{-1}.x)$ for every $x\in U_{\alpha}$.

    Let $x\in U_{\alpha}$ and suppose that $a(\alpha^{-1},x)=\mu\nu^{-1}$ as above. For the continuity of $a(\alpha^{-1},\cdot)$ at $x$, consider the set $W=\phi^{-1}(V_\nu)\cap\varphi_{\alpha}(\phi^{-1}(V_\mu\cap U_{\alpha^{-1}}))\cap k^{-1}(k(x))\cap l^{-1}(l(x))$. Since $\phi$, $k$ and $l$ are continuous and $\varphi_{\alpha}$ is open, we have that $W$ is an open neighborhood of $x$. Let $y\in W$ and observe that $y\in U_\alpha$, $\phi(y)\in V_{\nu}$ and there exists $z\in \phi^{-1}(V_\mu)\cap U_{\alpha^{-1}}$ such that $y=\varphi_{\alpha}(z)$. In particular $\phi(\sigma(y))=\phi(z)\in V_{\mu}$. Moreover, since $k(y)=k(x)=|\mu|$ and $l(y)=l(x)=|\nu|$, we have that $a(\alpha^{-1},y)=\mu\nu^{-1}$. Hence $a(\alpha^{-1},W)=\{\mu\nu^{-1}\}$ and $a(\alpha^{-1},\cdot)$ is continuous at $x$. Now, the equality $\tau^{l(x)}(\phi(x))=\tau^{k(x)}(\phi(\sigma(x)))$ implies that $\nu^{-1}.\phi(x)=\mu^{-1}.\phi(\alpha^{-1}.x)$. Note that there exists $\mu',\nu',\zeta\in\mathbb{F}'$ such that $\mu=\mu'\zeta$, $\nu=\nu'\zeta$ and $\mu'\nu'^{-1}$ is in reduced form. Using that the partial action $\varphi'$ is semi-saturated and that $\varphi'_{\zeta^{-1}}$ is injective, we then conclude that $\nu'^{-1}.\phi(x)=\mu'^{-1}.\varphi(\alpha^{-1}.x)$, from where it follows that $\phi(x)\in V_{\nu'\mu'^{-1}}=V_{a(\alpha^{-1},x)^{-1}}$ and
    \[a(\alpha^{-1},x).\phi(x)=(\mu'\nu'^{-1}).\phi(x)=\mu'(\nu'^{-1}.\phi(x))=\mu'.(\mu'^{-1}.\phi(\alpha.x))=\phi(\alpha.x).\]

    Using the above construction, for $\alpha\in\CA$ and $x\in U_{\alpha^{-1}}$, we define $a(\alpha,x)=a(\alpha^{-1},\alpha.x)^{-1}$. It is straightforward to check that $a(\alpha,\cdot):U_{\alpha^{-1}}\to\mathbb{F}'$ is continuous, $\phi(x)\in V_{a(\alpha,x)^{-1}}$ and $a(\alpha,x).\phi(x)=\phi(\alpha.x)$ for every $x\in U_{\alpha^{-1}}$. Inductively, we define $a(g,x)$ in such way that if $g=g_1\cdots g_{n+1}$ is the reduced form of $g$, then $a(g,x):=a(g_1,g_2\cdots g_{n+1}.x)a(g_2\cdots g_{n+1},x)$. To see that this expression is well-defined, we use that the partial action is semi-saturated to guarantee that if $x\in U_{g^{-1}}$, then $x\in U_{(g_2\cdots g_{n+1})^{-1}}$, since $\varphi_g=\varphi_{g_1}\circ\varphi_{g_2\cdots g_{n+1}}$ as functions.
    
    By induction, we prove the following: for every $n\in\mathbb{N}$, every $g\in\mathbb{F}$ with $|g|=n$, and every $x\in U_{g^{-1}}$, the map $a(g,\cdot):U_{g^{-1}}\to \mathbb{F}'$ is continuous, $\phi(x)\in V_{a(g,x)^{-1}}$ and $\phi(g.x)=a(g,x).\phi(x)$.

    The case $n=0$ is immediate.

    Suppose the statement is true for $n\in\mathbb{N}$. Let $g'\in\mathbb{F}$ with $|g'|=n+1$ and $x\in U_{g'^{-1}}$. Write $g'=gh$ where $|g|=1$ and $|h|=n$. By definition $a(gh,x)=a(g,h.x)a(h,x)$. Observe that $\phi(x)\in V_{a(h,x)^{-1}}$ and $\phi(h.x)\in V_{a(g,h.x)^{-1}}$. Since $\phi(h.x)=a(h,x).\phi(x)$, by the definition of partial action, we have that $\phi(x)\in V_{a(h,x)^{-1}a(g,h.x)^{-1}}=V_{a(g',x)^{-1}}$. Also, using that $\varphi$ is semi-saturated and that $\phi(x)\in V_{a(h,x)^{-1}}\cap V_{a(g',x)^{-1}}$, we see that \[\phi(g'.x)=\phi(g.(h.x))=a(g,h.x).\phi(h.x)=a(g,h.x).(a(h,x).\phi(x))=a(g',x).\phi(x).\]
    For the continuity of $a(g',\cdot)$, let $x\in U_{g'^{-1}}$, $U_1$ be an open neighborhood of $h.x$ such that $a(g,\cdot)$ is constant on $U_1$ and $U_2$ be an open neighborhood of $x$ such that $a(h,\cdot)$ is constant on $U_2$. Then, it is easy to see that $a(g',\cdot)$ is constant in the open neighborhood $\varphi_{h^{-1}}(U_1)\cap U_2$ of $x$. Similarly, we build a function $b$ such that $(\phi,a,b)$ is a continuous orbit equivalence.
\end{proof}

\subsection{Eventual conjugacy} Let $\mathbb{F}\curvearrowright X$  be an orthogonal semi-saturated partial dynamical system. By \cite[Theorem 4.3]{ExelLaca2003}, there exists a unique strongly continuous one-parameter group $\rho^X$ of automorphisms of 
$C_0(X) \rtimes \mathbb{F}$ such that 
\[\rho_t^X(f\dt_\af)=e^{it}f \dt_\af ~\text{and}~ \rho_t^X(g\dt_\emptyset)=g \dt_\emptyset \]
for all $t \in \mathbb{R}$, $\af \in \CA$, $f \in C_0(U_\af)$ and $g \in C_0(U_\emptyset)$. Since $\rho^X_t$ is $2\pi$-periodic, it induces  an action  of the circle, namely $\rho^X: \mathbb{T} \to \operatorname{Aut}(C_0(X) \rtimes \mathbb{F})$, such that 
\[\rho_z^X(f\dt_\af)=zf \dt_\af ~\text{and}~ \rho_z^X(g\dt_\emptyset)=g \dt_\emptyset \]
for all $z \in \mathbb{T}$, $\af \in \CA$, $f \in  C_0(U_\af)$ and $g \in C_0(U_\emptyset)$. 
In this context, we study when the isomorphism of crossed products in Corollary \ref{isom-equivalences:oss}(4) intertwines the actions introduced above. 

There is also an action $\gamma^X: \T \to \text{Aut}(C^*(\Gamma(X,\sigma)))$
such that $\gamma^X_z(f)(x,n,y)=z^nf(x,n,y)$ for all $z \in \T$, $(x,n,y) \in \Gamma(X, \sigma)$ and $f \in C_c(\Gamma(X, \sigma))$.

We first prove that there is a diagonal-preserving isomorphism between the groupoid $C^*$-algebra 
$ C^*( \Gamma(X, \sigma))$ and the crossed product $ C_0(X)\rtimes \mathbb{F}$ such that the groupoid action $\gamma^X_z$ is intertwined with the action $\rho^X_z$ by the isomorphism.

\begin{lem}\label{commute isom from DRG to CP} There is an isomorphism $\Pi_X: C^*( \Gamma(X, \sigma))   \to  C_0(X)\rtimes \mathbb{F} $  such that $\Pi_X(C_0(X))=C_0(X)$ and
$\Pi_X \circ \gamma^X_z = \rho^X_z \circ \Pi_X$ for all $z \in \T$. Moreover,  $\Pi_X^{-1}(C_0(X))=C_0(X)$ and $\gamma^X_{z} \circ \Pi_X^{-1} =\Pi_X^{-1} \circ  \rho^X_{z}$ for all $z \in \T$.
\end{lem}

\begin{proof} By Theorem \ref{thm: from PA to DRS}, there is an isomorphism $\widehat{\Xi}: C^*(\Gamma(X, \sigma)) \to C^*(\mathbb{F} \ltimes X)$ given by $\widehat{\Xi}(f)=f \circ \Xi$ for $f \in C_c(\Gamma(X,\sigma))$. Also, by \cite[Proposition 2.2]{Li2017}, there is an isomorphism $\Psi: C^*(\mathbb{F} \ltimes X) \to   C_0(X)\rtimes \mathbb{F} $ given by $\Psi(f)=\sum_{g} f(g, g^{-1}.U)\delta_g$, where $f(g, g^{-1}.U): U_g \to \C$ is defined by $x \mapsto f(g, g^{-1}.x)$, for $f \in C_c(\mathbb{F} \ltimes X)$. Then, we have an isomorphism $\Psi \circ \widehat{\Xi}  : C^*(\Gamma(X, \sigma)) \to  C_0(X)\rtimes \mathbb{F}$ given by $\Psi \circ \widehat{\Xi} (f)=\Psi(f \circ \Xi)=\sum_{g} (f \circ \Xi)(g, g^{-1}.U) \delta_g $. Clearly, $\Psi \circ \widehat{\Xi} (C_0(X))=C_0(X)$. We claim that $(\Psi \circ \widehat{\Xi}) \circ \gamma^X_z = \rho^X_z \circ (\Psi \circ \widehat{\Xi})$ for all $z \in \T$. Fix $z \in \T $. 
For $f \in C_c(\Gamma(X, \sigma))$, we have $$((\Psi \circ \widehat{\Xi}) \circ \gamma^X_z)(f)=\psi( \widehat{\Xi}(\gamma^X_z(f)))=\sum_{\af\bt^{-1} \in\mathbb{F}}  \widehat{\Xi}(\gamma^X_z(f))(\af\bt^{-1}, \bt\af^{-1}.U)\delta_{\af\bt^{-1}},$$
where 
\begin{align*}  \widehat{\Xi}(\gamma^X_z(f))(\af\bt^{-1}, \bt\af^{-1}.x)
&=(\gamma^X_z(f) \circ \Xi) (\af\bt^{-1}, \bt\af^{-1}.x)\\&=\gamma^X_z(f)(\varphi_{\af\bt^{-1}}(\bt\af^{-1}.x), |\af|-|\bt|, \bt\af^{-1}.x)\\
&=z^{|\af|-|\bt|}f(x, |\af|-|\bt|, \bt\af^{-1}.x)
\end{align*}
for $x \in U_{\af\bt^{-1}}$. 
On the other hand,  since 
$$ (\Psi \circ \widehat{\Xi} )(f)=\sum_{\af\bt^{-1} \in\mathbb{F}} \widehat{\Xi}(f)(\af\bt^{-1}, \bt\af^{-1}.U)\delta_{\af\bt^{-1}}$$ for $f \in C_c(\Gamma(X, \sigma))$, we have 
\begin{align*}\rho^X_z((\Psi \circ \widehat{\Xi})(f))
&=\sum_{\af\bt^{-1} \in\mathbb{F}} \rho^X_z \big( \widehat{\Xi}(f)(\af\bt^{-1}, \bt\af^{-1}.U)\delta_{\af\bt^{-1}}\big) \\
& =\sum_{\af\bt^{-1} \in\mathbb{F}} z^{|\af|-|\bt|}\big( \widehat{\Xi}(f)(\af\bt^{-1}, \bt\af^{-1}.U)\big)\delta_{\af\bt^{-1}},
\end{align*}
where  \begin{align*} \widehat{\Xi}(f)(\af\bt^{-1}, \bt\af^{-1}.x)&=(f \circ \Xi)(\af\bt^{-1}, \bt\af^{-1}.x)\\
&=f(\varphi_{\af\bt^{-1}}(\bt\af^{-1}.x), |\af|-|\bt|, \bt\af^{-1}.x) \\
&=f(x, |\af|-|\bt|, \bt\af^{-1}.x)
\end{align*}
for each $x \in U_{\af\bt^{-1}}$.
Thus, $(\Psi \circ \widehat{\Xi}) \circ \gamma^X_z = \rho^X_z \circ (\Psi \circ \widehat{\Xi})$ for all $z \in \T$.

Put $\Pi_X=\Psi \circ \widehat{\Xi}$.  From the above equation,  we then have $\Pi_X \circ \gamma^X_{\bar{z}} = \sm^X_{\bar{z}} \circ \Pi_X$ for all $z \in \T$. Then, $$\gamma^X_{z} \circ \Pi_X^{-1} =(\Pi_X \circ \gamma^X_{\bar{z}})^{-1} = (\sm^X_{\bar{z}} \circ \Pi_X)^{-1}=\Pi_X^{-1} \circ  \sm^X_{z}$$
 for all $z \in \T$. Lastly, it is clear that $\Pi_X^{-1}(C_0(X))=C_0(X)$. So, we are done.  
\end{proof}

Given  a free group $\mathbb{F}$ generated by a set $\CA$, we define the cocycle
 $\ell_{\mathbb{F}}:\mathbb{F}\to\mathbb{Z}$ to be the group homomorphism such that $\ell_{\mathbb{F}}(a)=1$ for every generator $a \in \CA$.

\begin{dfn}\label{dfn:evetually conjugay of PDS}  Let $\mathbb{F}\curvearrowright X$ and $\mathbb{F}'\curvearrowright Y$ be orthogonal semi-saturated partial dynamical systems.  We say that $\mathbb{F}\curvearrowright X$ and $\mathbb{F}'\curvearrowright Y$  are {\it eventually conjugate} if there is an isomorphism $(\phi,a)$ between $\mathbb{F}\curvearrowright X$ and $\mathbb{F}'\curvearrowright Y$  such that $\ell_{\mathbb{F}'}(a(g,x))=\ell_{\mathbb{F}}(g)$ for all $g\in \mathbb{F}$ and $x\in X$ such that $x\in U_{g^{-1}}$.
\end{dfn}

In the next theorem, we characterize eventual conjugacy in terms of the transformation groupoids, the crossed products, and the corresponding Deaconu-Renault systems.

 We recall that for a Deaconu-Renault system $(X, \sigma)$, a continuous cocycle $c_X: \Gamma(X, \sigma) \to \Z$ is defined by $c_X(x,n,y)=n$.

\begin{cor}\label{equivalences ec} Let $\mathbb{F}\curvearrowright X$ and $\mathbb{F}'\curvearrowright Y$ be orthogonal semi-saturated partial dynamical systems  and $(X,\sigma)$, $(Y,\tau)$ the corresponding Deaconu-Renault systems.  Then the following are equivalent:
\begin{enumerate}
\item $\mathbb{F}\curvearrowright X$ and $\mathbb{F}'\curvearrowright Y$  are eventually conjugate.
\item There is an isomorphism $\Theta: \mathbb{F} \ltimes X \to \mathbb{F}' \ltimes Y$, such that $\ell_{\mathbb{F}'}(\pi_1(\Theta(g,x)))=\ell_{\mathbb{F}}(g)$ for all $(g,x)\in \mathbb{F} \ltimes X$, where $\pi_1$ is the projection in the first coordinate.
\item There exists a continuous orbit equivalence $(\phi,a,b)$ between $\mathbb{F}\curvearrowright X$ and $\mathbb{F}'\curvearrowright Y$  such that $a$ is a cocycle, $(\phi,a)$ preserves essential stabilizers and $\ell_{\mathbb{F}'}(a(g,x))=\ell_{\mathbb{F}}(g)$ for all $g\in \mathbb{F}$ and $x\in X$ such that $x\in U_{g^{-1}}$.
\item There is an isomorphism $\Phi: C_0(X)\rtimes \mathbb{F} \to C_0(Y)\rtimes \mathbb{F}' $ such that $\Phi(C_0(X))=C_0(Y)$ and $\Phi \circ \rho_z^X=\rho_z^Y \circ \Phi$.
\item $(X, \sigma)$ and $(Y, \tau)$ are eventually conjugate.
\item There is an isomorphism $\Theta':\Gamma(X, \sigma)\to \Gamma(Y, \tau)$ such that $c_X=c_Y\circ\Theta'$.
\item There is an isomorphism $\Phi':C^*( \Gamma(X, \sigma)) \to C^*(\Gamma(Y, \tau))$ such that $\Phi'(C_0(X))=C_0(Y)$ and $\Phi' \circ \gamma_z^X=\gamma_z^Y \circ \Phi'$.
\end{enumerate}
\end{cor}

\begin{proof}

(1)$\iff$(2)$\iff$(3):
It follows from Theorem \ref{isom-equivalences} and Definition \ref{dfn:evetually conjugay of PDS}.

(2)$\implies$(6): Choose $(x,k,y)\in\Gamma(X,\sigma)$, where  $x\in\operatorname{dom}(\sigma^n)$, $y\in\operatorname{dom}(\sigma^m)$, $k=n-m$ and $\sigma^n(x)=\sigma^m(y)$ for some $n,m\in\mathbb{N}$. There then exist $\alpha,\beta\in\mathbb{F}_+$ such that $|\alpha|=n$, $|\beta|=m$, $x\in U_{\alpha}$ and $y\in U_{\beta}$. Assume that the reduced form of $\alpha\beta^{-1}$ is $\alpha'\beta'^{-1}$, so that there exists $\delta\in\mathbb{F}_+$ such that $\alpha=\alpha'\delta$ and $\beta=\beta'\delta$. We then define $\Theta': \Gamma(X, \sigma)\to \Gamma(Y, \tau)$ by \begin{align*}\Theta'(x,k,y)&=\Xi\circ \Theta \circ \Xi^{-1}(x,k,y) \\
&=\Xi \circ \Theta (\af'\bt'^{-1},y)\\
&=\Xi(a(\af'\bt'^{-1}), \phi(y))\\
&=((\af'\bt'^{-1}). \phi(y), |\af'|-|\bt'|, \phi(y))
\end{align*} for $(x,k,y) \in \Gamma(X, \sigma)$.
Observe that $$c_X(x,k,y)=|\af|-|\bt|=|\af'|-|\bt'|=c_Y \circ \Theta'(x,k,y)$$ for $(x,k,y) \in \Gamma(X, \sigma)$. Thus, $\Theta':\Gamma(X, \sigma)\to \Gamma(Y, \tau)$ is an isomorphism such that $c_X=c_Y\circ\Theta'$.

(6)$\implies$(2): Define $\Theta: \mathbb{F} \ltimes X \to \mathbb{F}' \ltimes Y$  by $\Theta=\Xi^{-1} \circ \Theta \circ \Xi$. Then, $\Theta$ is an isomorphism such that
\begin{align*} 
\Theta(\af\bt^{-1},x)&=\Xi^{-1} \circ \Theta \circ \Xi(\af\bt^{-1},x) \\
&=\Xi^{-1} \circ \Theta'(\varphi_{\af\bt^{-1}(x)}, |\af|-|\bt|, x)
\end{align*} for $(\af\bt^{-1},x) \in \mathbb{F} \ltimes X$. 
 Since $c_X=c_Y\circ\Theta'$, one sees that $l_{\mathbb{F}'}(\pi_1(\Theta(\af\bt^{-1}, x)))=l_{\mathbb{F}}(\af\bt^{-1},x)$ for $(\af\bt^{-1},x) \in \mathbb{F} \ltimes X$.
 
 (4)$\implies$(7): By Lemma \ref{commute isom from DRG to CP}, we have 
\begin{align*}       ( \Pi_Y^{-1}\circ\Phi\circ\Pi_X )\circ \gamma^X_z
 &=   ( \Pi_Y^{-1}\circ\Phi )\circ  (\Pi_X \circ \gamma^X_z )\\
 &=( \Pi_Y^{-1}\circ\Phi )\circ  ( \sm^X_z \circ  \Pi_X)\\
 &= \Pi_Y^{-1}\circ (\Phi \circ   \sm^X_z ) \circ  \Pi_X \\
 &=\Pi_Y^{-1}\circ (\sm_z^Y \circ \Phi) \circ  \Pi_X \\
 &=(\Pi_Y^{-1}\circ \sm_z^Y ) \circ (\Phi \circ  \Pi_X) \\
  &=( \gamma^Y_z \circ \Pi_Y^{-1} ) \circ (\Phi \circ  \Pi_X) \\
   &= \gamma^Y_z \circ  (\Pi_Y^{-1}  \circ \Phi \circ  \Pi_X )
\end{align*} 
for all $z \in \T$. Also, $( \Pi_Y^{-1}\circ\Phi\circ\Pi_X )(C_0(X))=C_0(Y)$. 

(7)$\implies$(4) : It is analogous to (4)$\implies$(7).

(5)$\iff$(6)$\iff$(7): It follows from \cite[Theorem 8.10]{CRST}.
\end{proof}

\section{An application to generalized Boolean dynamical systems}\label{Sec 5}

A \emph{Boolean algebra}  $\CB$ is a relatively complemented distributive lattice with the least element $\emptyset$.  This type of Boolean algebra is sometimes referred to as a \emph{generalized Boolean algebra}.
For $A, B \in \CB$, the {\it meet} of $A$ and $B$ is denoted by $A \cap B$, the {\it join} of $A$ and $B$ is denoted by $A \cup B$, and the 
{\it relative complement} of $A$ relative to  $B$ is denoted by 
$B \setminus A$.   The Boolean algebra $\CB$ is called {\em unital} if there exists $1 \in \CB$ such that $1 \cup A = 1$ and $1 \cap A=A$ for all $A \in \CB$. 
The partial order is defined by 
$$ A\subseteq B ~\text{if and only if}~ A \cap B =A$$
for $A, B \in \CB$. 
We say that   $A$ is  a  \emph{subset of} $B$ if $A \subseteq B$. Given $X,Y \subseteq \CB$, we write
$$\uparrow X:=\{B \in \CB: A \subseteq B ~\text{for some}~A \in X\}$$
and $\uparrow_Y X:= Y \cap \uparrow X$.

  A non-empty subset $\CI\subseteq \CB$ is called an {\it ideal} if 
  $A \cup B \in \CI$ whenever $A, B \in \CI$, and it is a lower set, that is, if $A \in \CI$ and $B \subseteq A$, then $B \in \CI$. 
   Every ideal of $\CB$ is again a Boolean algebra. 
     A non-empty subset $\eta \subseteq \CB$ is called a {\em filter} if $\emptyset \notin \eta$, $A \cap B \in \eta$ whenever $A,B \in \eta$, and it is an  upper set, that is, if $ A \in \eta$ and  $A \subseteq B$, then $B \in \eta$. 
 An {\it ultrafilter} is a filter that is maximal in the set of filters with respect to inclusion.
 We denote  the set of all ultrafilters of $\CB$ by $\widehat{\CB}$. 
  For $A\in\CB$, let $Z(A):=\{\eta\in\widehat{\CB}:A\in\eta\}$. We equip $\widehat{\CB}$ with the topology generated by $\{Z(A): A\in\CB\}$. Then, 
$\widehat{\CB}$
  is a totally disconnected, locally compact Hausdorff space in which each 
$Z(A)$ is compact and open.

  A {\em generalized Boolean dynamical system} is defined as a quadruple $(\CB, \CL, \theta, \CI_\af)$, where $(\CB, \CL, \theta)$ forms a Boolean dynamical system, and $\{\CI_\af\}_{\af \in \CL}$ represents a family of ideals in $\CB$ such that $\theta_\af(\CB) \subseteq \CI_\af$ for each $\af \in \CL$. When we need to specify the Boolean dynamical system $(\CB, \CL, \theta)$, we denote $\CI_\af$ as $\CI_\af^{(\CB, \CL, \theta)}$.  We refer the reader to \cite{CaK2, CasK1, CasK2} for  more details.

Let $(\CB, \CL,\theta, \CI_\af)$ be a generalized Boolean dynamical system and let
  \begin{align*}
 E^0&:=\widehat{\CB}, \\
 F^0&:=\widehat{\CB} \cup\{\emptyset\}, \\
 E^1&:=\bigl\{e^\alpha_\eta:\alpha\in\CL ~\text{and}~\eta\in \widehat{\CI_\alpha}\bigr\}.
 \end{align*}
We 
 equip $E^0$ with the topology given by the basis $\{Z(A):A\in\CB\}$ and 
   $E^1$ with the topology generated by 
$\bigcup_{\alpha\in\CL} \{Z^1(\af, B):B\in\CI_\alpha\}, $
where $$Z^1(\af, B):=\{e^\af_\eta: \eta \in \widehat{\CI_\alpha}  , B \in \eta\}.$$  Note that $E^1_{(\CB,\CL,\theta,\CI_\alpha)}$ is homeomorphic to the disjoint union of the family $\{\widehat{\CI_\alpha}\}_{\af\in\CL}$, where we 
equip $\widehat{\CI_\alpha}$ with the topology generated by $\{Z(\af, A): A\in\CI_\af\}$, where 
  we let $Z(\af, A):=\{\eta \in \widehat{\CI_\alpha}: A \in \eta\}$ for $A \in \CI_\af$. 
We also  equip  $F^0$  with a suitable topology;  if $\CB$ possesses a unit element, the topology is such that $\{\emptyset\}$ is treated as an isolated point.
  If $\CB$ does not have a unit element, $\emptyset$ plays the role of infinity in the one-point compactification of $X_\emptyset$.

For each $\alpha \in \CL$, we define $\widehat{\theta}_\alpha:\widehat{\CI_\alpha} \to \widehat{\CB} \cup\{\emptyset\}$ by $$\widehat{\theta}_\alpha(\eta)=\{A \in \CB:\theta_{\alpha}(A) \in \eta\}.$$
Note that each $\widehat{\theta}_\alpha$ is continuous (\cite[Lemma 2.3]{CasK1}).
We  define the maps
 $d:E^1\to E^0$ and $r:E^1\to F^0$ by 
\begin{align*}
	d(e^\alpha_\eta)&=\uparrow_{\CB} \eta =\{B \in \CB: A \subseteq B ~\text{for some}~  A \in \eta\} \text{ and } \\ r(e^\alpha_\eta)&=\widehat{\theta}_{\alpha}(\eta)=\{A \in \CB:\theta_{\alpha}(A) \in \eta\}.
\end{align*}
 Then, $(E^1, d,r)$ is a topological correspondence from $E^0$ to $F^0$ (\cite[Proposition 7.1]{CasK1}).
Given $n \geq 2$,  we define the space of paths of length $n$ as
\[E^{n}:=\{(e^{\af_1}_{\eta_1},\ldots,e^{\af_n}_{\eta_n}) \in \prod_{i=1}^n E^1:d(e^{\af_i}_{\eta_i})=r(e^{\af_{i+1}}_{\eta_{i+1}}) ~\text{for}~1\leq i<n\},\]
equipped with the subspace topology inherited from 
 the product space 
  $\prod_{i=1}^n E^1$.  
We define the {\it finite path space} as 
 $E^*:= \sqcup_{n=0}^{\infty} E^n$, equipped with the disjoint union topology. The {\it infinite path space} is defined by
\[E^{\infty}:=\{(e^{\af_i}_{\eta_i})_{i\in\mathbb{N}}\in \prod_{i=1}^\infty E^1:d(e^{\af_i}_{\eta_i})=r(e^{\af_{i+1}}_{\eta_{i+1}}) ~\text{for}~ i\in\mathbb{N}\}.\]
 The {\it boundary path space} of $E$ is defined by 
$$\partial E :=E^{\infty} \sqcup \{(e_k)_{k=1}^n \in E^* : d((e_k)_{k=1}^n )  \in E^0_{sg}\}.$$
We denote by $\sm_E:\partial E\setminus E^0_{sg}\to \partial E$ the shift map that removes the first edge for paths of length greater or equal to 2. For elements $\mu$ of length 1, $\sm_E(\mu)=d(\mu)$.
For a subset $S \subset E^*$, denote by 
$$\CZ(S)=\{\mu \in \partial E:~\text{either} ~r(\mu) \in S, ~\text{or there exists}~  1 \leq i \leq |\mu| ~\text{such that}~ \mu_1 \cdots \mu_i \in 
S \}.$$
We endow $\partial E$ with the topology generated by the basic open sets $\CZ(U)\cap \CZ(K)^c$, where $U$ is an open set of $E^*$ and $K$ is a compact set of $E^*$.
Note that $\partial E$ is a locally compact Hausdorff space and that $\sm_E$ is a local homeomorphism. So, we have a Deaconu-Renault system $(\partial E, \sm_E)$.

Given a generalized Boolean dynamical system $(\CB, \CL,\theta, \CI_\af)$, we let  $\mathbb{F}_\CL$ be the free group generated by $\CL$. It then is known that there is a semi-saturated orthogonal  partial action  $\varphi:=(\{U_t\}_{t \in \mathbb{F}}, \{\varphi_t\}_{t \in \mathbb{F}_\CL})$ of  $\mathbb{F}_\CL$ on $\partial E$ (\cite[Proposition 3.6]{CasK1}). 

	 A 
	\emph{$(\CB,\CL,\theta, \CI_\af)$-representation}  (\cite[Definition 3.3]{CaK2}) in a $C^*$-algebra $A $
		is   families of projections $\{P_A:A\in\mathcal{B}\}$ and  partial isometries $\{S_{\alpha,B}:\alpha\in\mathcal{L},\ B\in\mathcal{I}_\alpha\}$ 
	that satisfy  
	\begin{enumerate}
		\item[(i)] $P_\emptyset=0$, $P_{A\cap A'}=p_Ap_{A'}$, and $P_{A\cup A'}=P_A+P_{A'}-P_{A\cap A'}$ for $A,A'\in\mathcal{B}$;
		\item[(ii)] $P_AS_{\alpha,B}=S_{\alpha,  B}P_{\theta_\af(A)}$  for $A\in\mathcal{B}$, $\alpha \in\mathcal{L}$ and $B\in\mathcal{I}_\alpha$;
		\item[(iii)] $S_{\alpha,B}^*S_{\alpha',B'}=\delta_{\alpha,\alpha'}P_{B\cap B'}$ for  $\alpha,\alpha'\in\mathcal{L}$, $B\in\mathcal{I}_\alpha$ and $B'\in\mathcal{I}_{\alpha'}$;
		\item[(iv)] $P_A=\sum_{\af \in\Delta_A}S_{\af,\theta_\af(A)}S_{\af,\theta_\af(A)}^*$ for   $A\in \mathcal{B}_{reg}$. 
	\end{enumerate}
It is established in \cite[Theorem 5.5]{CaK2} that for any given generalized Boolean dynamical system $(\CB, \CL, \theta, \CI_\alpha)$, there exists a universal $(\CB, \CL, \theta, \CI_\alpha)$-representation $\{p_A, s_{\alpha,B} : A \in \CB, \alpha \in \CL, \text{ and } B \in \CI_\alpha\}$.
We denote by $C^*(\CB, \CL,\theta, \CI_\af)$ the $C^*$-algebra generated by a universal 
$(\CB,\CL,\theta, \CI_\af)$-representation $\{p_A, s_{\af,B}\}$.

The {\it diagonal subalgebra $D(\CB, \CL,\theta, \CI_\af)$} of $C^*(\CB, \CL,\theta, \CI_\af)$ is defined as the subalgebra of $C^*(\CB, \CL,\theta, \CI_\af)$ generated by the commuting projections $s_{\af,A}s_{\af,A}^*$:
$$D(\CB, \CL,\theta, \CI_\af)
=C^*(\{s_{\af,A}s_{\af,A}^*: \af \in \CL^* ~\text{and}~ A \in \CI_\af\}), $$
which can also be expressed as
$$D(\CB, \CL,\theta, \CI_\af)=\overline{span}\{s_{\af,A}s_{\af,A}^*: \af \in \CL^* ~\text{and}~ A \in \CI_\af\}.$$
We then have the following.

\begin{thm}\label{thm1}Let $(\CB, \CL,\theta, \CI_\af^{(\CB, \CL,\theta)})$ and $(\CB', \CF, \theta', \CI_a^{(\CB', \CF, \theta')})$ be generalized Boolean dynamical systems such that $\CB$, $\CL$, $\CB'$ and $\CF$ are countable. Let $\mathbb{F}_{\CL} \curvearrowright \partial E$ and $\mathbb{F}_{\CF} \curvearrowright \partial F$ be the corresponding partial dynamical systems,  $(\partial E, \sm_E)$ and $(\partial F, \sm_F)$ be the corresponding Deaconu-Renault systems. Then the following are equivalent:
\begin{enumerate}
\item $\mathbb{F}_{\CL} \curvearrowright \partial E$ and $\mathbb{F}_{\CF} \curvearrowright \partial F$   are isomorphic in $\PDS$.
\item $ \mathbb{F}_{\CL} \ltimes \partial E$ and $ \mathbb{F}_{\CF} \ltimes \partial F$ are isomorphic as topological groupoids.
\item There exists a continuous orbit equivalence $(\phi,a,b)$ between $\mathbb{F}_{\CL} \curvearrowright \partial E$ and $\mathbb{F}_{\CF} \curvearrowright \partial F$  such that $a$ is a cocycle and $(\phi,a)$ preserves essential stabilisers.
\item There is an isomorphism $\Phi: C_0(\partial E)\rtimes \mathbb{F}_{\CL} \to C_0(\partial F)\rtimes \mathbb{F}_{\CF} $  such that $\Phi(C_0(\partial E))=C_0(\partial F)$.
\item There is an essential-stabiliser-preserving continuous orbit equivalence  from $(\partial E, \sm_E)$ to $(\partial F, \sm_F)$.
\item $\Gamma(\partial E, \sm_E)$ and $ \Gamma(\partial F, \sm_F)$ are isomorphic as topological groupoids.
\item There is an isomorphism $\Phi':C^*( \Gamma(\partial E, \sm_E)) \to C^*(\Gamma(\partial F, \sm_F))$ such that $\Phi'(C_0(\partial E))=C_0(\partial F)$.
\item There is an isomorphism $\pi: C^*(\CB, \CL,\theta, \CI_\af^{(\CB, \CL,\theta)}) \to C^*(\CB', \CF, \theta', \CI_a^{(\CB', \CF, \theta')})$ with $\pi(D(\CB, \CL,\theta, \CI_\af^{(\CB, \CL,\theta)}))=D(\CB', \CF, \theta', \CI_a^{(\CB', \CF, \theta')})$. 
\end{enumerate}
\end{thm}

\begin{proof}(1)-(7) are equivalent by Corollary\ref{isom-equivalences:oss}.  
(7)$\iff$(8) follows form  \cite[Corollary 7.11]{CasK1} and \cite[Theorem 3.5]{Kang}. 
\end{proof}

\begin{remark}
    In the context of Section~\ref{Sec 4}, if the topological space $X$ is a Stone space, and each $U_g$ is a clopen set, then, by \cite[Theorem~5.1]{CasK2} we can construct a generalized Boolean dynamical system such that the transformation groupoid of the original action is isomorphic to the transformation groupoids coming from the generalized Boolean dynamical system.
\end{remark}

By the universal property of $C^*(\CB,\CL,\theta,\CI_\alpha)=C^*(p_A, s_{\alpha,B})$, there is a strongly continuous action $\tau:\mathbb T\to {\rm Aut}(C^*(\CB,\CL,\theta, \CI_\alpha))$, which we call the {\it gauge action}, such that
\[
\tau_z(p_A)=p_A   \ \text{ and } \ \tau_z(s_{\alpha,B})=zs_{\alpha,B}
\]
for $A\in \CB$, $\alpha \in \CL$ and $B \in \CI_\alpha$. We say that an ideal $I$ of $C^*(\CB, \CL, \theta, \CI_\alpha)$ is \emph{gauge-invariant} if $\tau_z(I)=I$ for every $z\in\T$.

\begin{thm}\label{equivalences ec:GDBS} Let $(\CB, \CL,\theta, \CI_\af^{(\CB, \CL,\theta)})$ and $(\CB', \CF, \theta', \CI_a^{(\CB', \CF, \theta')})$ be generalized Boolean dynamical systems such that $\CB$, $\CL$, $\CB'$ and $\CF$ are countable. Let $\mathbb{F}_{\CL} \curvearrowright \partial E$ and $\mathbb{F}_{\CF} \curvearrowright \partial F$ be the corresponding partial dynamical systems,  $(\partial E, \sm_E)$ and $(\partial F, \sm_F)$ be the corresponding Deaconu-Renault systems. Then the following are equivalent:

\begin{enumerate}
\item $\mathbb{F}_{\CL} \curvearrowright \partial E$ and $\mathbb{F}_{\CF} \curvearrowright \partial F$   are  eventually conjugate.
\item There is an isomorphism $\Theta:  \mathbb{F}_{\CL} \ltimes \partial E \to \mathbb{F}_{\CF} \ltimes \partial F$, such that $\ell_{\mathbb{F}_{\CF}}(\pi_1(\Theta(g,x)))=\ell_{\mathbb{F}_{\CL}}(g)$ for all $(g,x)\in \mathbb{F}_{\CL} \ltimes \partial E$, where $\pi_1$ is the projection in the first coordinate.
\item There exists a continuous orbit equivalence $(\phi,a,b)$ between $\mathbb{F}_{\CL} \curvearrowright \partial E$ and $\mathbb{F}_{\CF} \curvearrowright \partial F$  such that $a$ is a cocycle, $(\phi,a)$ preserves essential stabilizers and $\ell_{\mathbb{F}_{\CF}}(a(g,x))=\ell_{\mathbb{F}_\CL}(g)$ for all $g\in \mathbb{F}_{\CL}$ and $x\in  \partial E$ such that $x\in U_{g^{-1}}$.
\item There is an isomorphism $\Phi: C_0(\partial E)\rtimes \mathbb{F}_{\CL} \to C_0(\partial F)\rtimes \mathbb{F}_{\CF} $  such that $\Phi(C_0(\partial E))=C_0(\partial F)$ and $\Phi \circ \rho_z^{\partial E}=\rho_z^{\partial F} \circ \Phi$.
\item $\Gamma(\partial E, \sm_E)$ and $ \Gamma(\partial F, \sm_F)$ are eventually conjugate.
\item There is an isomorphism $\Theta':\Gamma(\partial E, \sm_E)\to \Gamma(\partial F, \sm_F)$ such that $c_{\partial E}=c_{\partial F}\circ\Theta'$.
\item There is an isomorphism $\Phi':C^*( \Gamma(\partial E, \sm_E)) \to C^*(\Gamma(\partial F, \sm_F))$ such that $\Phi'(C_0(\partial E))=C_0(\partial F)$ and $\Phi' \circ \gamma_z^{\partial E}=\gamma_z^{\partial F}\circ \Phi'$.
\item There is an isomorphism $\pi: C^*(\CB, \CL,\theta, \CI_\af^{(\CB, \CL,\theta)}) \to C^*(\CB', \CF, \theta', \CI_a^{(\CB', \CF, \theta')})$ with $\pi(D(\CB, \CL,\theta, \CI_\af^{(\CB, \CL,\theta)}))=D(\CB', \CF, \theta', \CI_a^{(\CB', \CF, \theta')})$ and $\pi \circ \tau_z=\tau'_z \circ \pi$, where $\tau'$ is the gauge action on  $C^*(\CB', \CF, \theta', \CI_a^{(\CB', \CF, \theta')})$.
\end{enumerate}
\end{thm}


\begin{thebibliography}{10}

\bibitem{BCGW2023}
G. Boava, G. G. de Casto, D. Gonçalves and D. W. van Wyk,
{\em $C^*$-Algebras of one-sided subshifts over arbitrary alphabets},
arXiv:2312.17644 [math.OA] (2023). 

\bibitem{BT1998}
M. Boyle and J. Tomiyama, 
{\em Bounded topological orbit equivalence and $C^*$-algebras},
 J. Math. Soc. Japan \textbf{50}(1998), 317--329.

\bibitem{CRST} T. M. Carlsen, E. Ruiz, A. Sims and M. Tomforde, {\em Reconstruction of groupoids and $C^*$-rigidity of dynamical systems}, Adv. Math. \textbf{390}(2021), 107923.

\bibitem{CaK2} 
T. M. Carlsen and E. J. Kang, 
{\em Gauge-invariant ideals of $C^*$-algebras of  Boolean dynamical systems}, 
J. Math. Anal. Appl. \textbf{488}(2020), 124037.

\bibitem{CarlsenLarsen2016}
T. M. Carlsen and N. S. Larsen,
{\em Partial actions and KMS states on relative graph $C^*$-algebras},
J. Funct. Anal. \textbf{271}(2016), 2090--2132. 

\bibitem{CasK1}
 G. G. de Castro and E. J. Kang,  
{\em  Boundary path groupoids of generalized Boolean dynamical systems and their $C^*$-algebras}, 
J. Math. Anal. Appl. \textbf{518}(2023), 126662.

\bibitem{CasK2}
 G. G. de Castro and E. J. Kang, 
{\em $C^*$-algebras of generalized Boolean dynamical systems as partial crossed products}, J. Algebr. Comb. \textbf{58}(2023), 355--385. 

\bibitem{CastroWyk}
 G. G. de Castro and D. W. van Wyk,
{\em Labelled space $C^*$-algebras as partial crossed products and a simplicity characterization},
J. Math. Anal. Appl. \textbf{491}(2020), 124290.
\bibitem{Dok1}
M. Dokuchaev,
{\em Recent developments around partial actions}, 
São Paulo J. Math. Sci. \textbf{13}(2019), 195--247. 

\bibitem{ExelBook}
R. Exel,
{\em Partial dynamical systems, Fell bundles and applications},
  Math. Surv. Monogr. \textbf{224}(2017), American Mathematical Society, Providence, RI.

\bibitem{ExelLaca1999}
R. Exel and M.~Laca,
{\em Cuntz-Krieger algebras for infinite matrices},
J. Reine Angew. Math. \textbf{512}(1999), 119--172. 

\bibitem{ExelLaca2003}
R. Exel and M.~Laca,
{\em Partial dynamical systems and the KMS condition},
  Comm. Math. Phys. \textbf{232}(2003), 223--277.

\bibitem{GPS}
T. Giordano, I.F. Putnam and C.F. Skau, 
{\em Topological orbit equivalence and $C^*$-crossed products},
 J. Reine Angew. Math. \textbf{469}(1995), 51--111.

\bibitem{Kang} E. J. Kang, {\em A generalized uniqueness theorem for generalized Boolean dynamical systems}, J. Math. Anal. Appl. \textbf{538}(2024), 128381.

\bibitem{Li2018} X. Li, {\em Continuous orbit equivalence rigidity}, Ergod. Th. \& Dynam. Sys. \textbf{38}(2018), 1543--1563.

\bibitem{Li2017} X. Li, {\em Partial transformation groupoids attached to graphs and semigroups}, Int. Math. Res. Not. \textbf{17}(2017), 5233--5259.

\bibitem{T1996}
 J. Tomiyama, 
{\em Topological full groups and structure of normalizer in transformation group $C^*$-algebras},
 Pacific J. Math. \textbf{173}(1996), 571--583.

\end{thebibliography}
\end{document}